\documentclass[%
  letterpaper,
 twocolumn,
  colorlinks,
]{preprint}

\usepackage{amsmath}
\usepackage{amsthm}
\usepackage{amsfonts}
\usepackage{bbm}

\usepackage[noadjust]{cite}

\usepackage[english]{babel}

\usepackage{graphicx}
\usepackage{wrapfig}
\usepackage{enumerate,enumitem} 
\usepackage{caption}
\usepackage{subcaption}
\usepackage{tikz}

\usepackage{algorithm}
\usepackage{etoolbox}
\let\classAND\AND
\let\AND\relax
\usepackage{algorithmic}

\let\AND\classAND
\AtBeginEnvironment{algorithmic}{\let\AND\algoAND}
\floatname{algorithm}{Algorithm}

\usepackage{hyperref}    
\usepackage{nicefrac}
\usepackage{xspace}
\usepackage{mathtools}

\newcommand{\imunit}{{\dot{\imath\hspace*{-0.2em}\imath}}}
\newcommand{\trans}{\ensuremath{\mkern-1.5mu\mathsf{T}}}
\newcommand{\herm}{\ensuremath{\mathsf{H}}}
\def\Sys{\mathcal{G}}
\def\Sysred{\Sys_r}
\def\wt#1{\widetilde{#1}}
\def\wh#1{\widehat{#1}}

\def\hot#1{\textcolor{red}{#1}}

\mathcode`\@="8000
{\catcode`\@=\active \gdef@{\mkern1mu}}

\def\CH{\mathcal{H}}
\def\CX{\mathcal{X}}

\def\CL{\mathcal{L}}

\def\CW{\mathcal{W}}
\def\CM{\mathcal{M}}
\def\CN{\mathcal{N}}

\def\CY{\mathcal{Y}}

\def\CJ{\mathcal{J}}
\def\CK{\mathcal{K}}

\def\R{{\mathbbm R}}
\def\C{{\mathbbm C}}

\def\Cmm{\C^{m\times m}}
\def\Cnn{\C^{n\times n}}
\def\Cpm{\C^{p\times m}}

\def\Cpp{\C^{p\times p}}
\def\Cpr{\C^{p\times r}}
\def\Crm{\C^{r\times m}}
\def\Crr{\C^{r\times r}}
\def\Rn{\R^n}
\def\Rm{\R^m}
\def\Rp{\R^p}
\def\Rr{\R^r}

\def\Rmm{\R^{m\times m}}
\def\Rnn{\R^{n\times n}}
\def\Rnm{\R^{n\times m}}
\def\Rpn{\R^{p\times n}}
\def\Rpm{\R^{p\times m}}
\def\Rpp{\R^{p\times p}}
\def\Rrr{\R^{r\times r}}
\def\Rpr{\R^{p\times r}}
\def\Rrm{\R^{r\times m}}
\def\Rnr{\R^{n\times r}}

\newcommand{\BA}{\ensuremath{\mathbf{A}}}
\newcommand{\BB}{\ensuremath{\mathbf{B}}}
\newcommand{\BC}{\ensuremath{\mathbf{C}}}
\newcommand{\BD}{\ensuremath{\mathbf{D}}}
\newcommand{\BE}{\ensuremath{\mathbf{E}}}

\newcommand{\BG}{\ensuremath{\mathbf{G}}}

\newcommand{\BI}{\ensuremath{\mathbf{I}}}
\newcommand{\BJ}{\ensuremath{\mathbf{J}}}
\newcommand{\BK}{\ensuremath{\mathbf{K}}}
\newcommand{\BL}{\ensuremath{\mathbf{L}}}
\newcommand{\BM}{\ensuremath{\mathbf{M}}}
\newcommand{\BN}{\ensuremath{\mathbf{N}}}

\newcommand{\BP}{\ensuremath{\mathbf{P}}}
\newcommand{\BQ}{\ensuremath{\mathbf{Q}}}
\newcommand{\BR}{\ensuremath{\mathbf{R}}}

\newcommand{\BT}{\ensuremath{\mathbf{T}}}
\newcommand{\BU}{\ensuremath{\mathbf{U}}}
\newcommand{\BV}{\ensuremath{\mathbf{V}}}
\newcommand{\BW}{\ensuremath{\mathbf{W}}}
\newcommand{\BX}{\ensuremath{\mathbf{X}}}
\newcommand{\BY}{\ensuremath{\mathbf{Y}}}
\newcommand{\BZ}{\ensuremath{\mathbf{Z}}}

\newcommand{\Be}{\ensuremath{\mathbf{e}}}

\newcommand{\Bu}{\ensuremath{\mathbf{u}}}

\newcommand{\Bx}{\ensuremath{\mathbf{x}}}
\newcommand{\By}{\ensuremath{\mathbf{y}}}

\newcommand{\Bzero}{\ensuremath{\mathbf{0}}}
\newcommand{\BSigma}{\ensuremath{\mathbf{\Sigma}}}

\newcommand{\bbL}{\ensuremath{\wt{\mathbb{L}}}}
\newcommand{\bbM}{\ensuremath{\wt{\mathbb{M}}}}
\newcommand{\bbH}{\ensuremath{\wt{\mathbb{H}}}}
\newcommand{\bbG}{\ensuremath{\wt{\mathbb{G}}}}

\def\wtU{\wt{\BU}}
\def\wtL{\wt{\BL}}
\def\wtBA{\wt{\BA}}
\def\wtBB{\wt{\BB}}
\def\wtBC{\wt{\BC}}
\def\wtBZ{\wt{\BZ}}
\def\wtBY{\wt{\BY}}
\def\wtBSigma{\wt{\BSigma}}

\def\BGforEA{\BG_{\sigma,@\BA}}
\def\BGforB{\BG_\BB}
\def\BGforC{\BG_\BC}

\def\BQbst{\BQ_{\CW}}
\def\BQprbt{\BQ_{\CM}}
\def\BPprbt{\BP_{\CN}}
\def\BQbrbt{\BQ_{\CJ}}
\def\BPbrbt{\BP_{\CK}}
\def\BPgen{\BP_{\CX}}
\def\BQgen{\BQ_{\CY}}

\def\BRj{\BR_\CJ}
\def\BRk{\BR_\CK}
\def\hatBRj{\wh{\BR}_\CJ}
\def\hatBRk{\wh{\BR}_\CK}
\def\BCbrbt{\BC_\CJ}
\def\BBbrbt{\BB_\CK}
\def\hatBCbrbt{\wh{\BC}_\CJ}
\def\hatBBbrbt{\wh{\BB}_\CK}

\newcommand{\BAgenX}{\ensuremath{\BA}}
\newcommand{\BBgenX}{\ensuremath{\BB_{\CX}}}
\newcommand{\BCgenX}{\ensuremath{\BC_{\CX}}}
\newcommand{\BDgenX}{\ensuremath{\BD_{\CX}}}
\newcommand{\nx}{\ensuremath{n}}
\newcommand{\px}{\ensuremath{p_x}}
\newcommand{\mx}{\ensuremath{m_x}}

\newcommand{\BPgenX}{\ensuremath{\BP_{\CX}}}

\newcommand{\BUgenX}{\ensuremath{\BU_{\CX}}}

\newcommand{\BLgenYTilde}{\ensuremath{\wt{\BL}_{\CY}}}
\newcommand{\BUgenXTilde}{\ensuremath{\wt{\BU}_{\CX}}}
\newcommand{\BAgenY}{\ensuremath{\BA}}
\newcommand{\BBgenY}{\ensuremath{\BB_{\CY}}}
\newcommand{\BCgenY}{\ensuremath{\BC_{\CY}}}
\newcommand{\BDgenY}{\ensuremath{\BD_{\CY}}}
\newcommand{\ny}{\ensuremath{n}}
\newcommand{\py}{\ensuremath{p_y}}
\newcommand{\my}{\ensuremath{m_y}}

\newcommand{\BQgenY}{\ensuremath{\BQ_{\CY}}}

\newcommand{\BLgenY}{\ensuremath{\BL_{\CY}}}

\theoremstyle{plain}\newtheorem{theorem}{Theorem}
\theoremstyle{plain}
\theoremstyle{definition}\newtheorem{definition}{Definition}
\theoremstyle{definition}\newtheorem{remark}{Remark}

\newcommand{\MOR}{\textsf{MOR}\xspace}
\newcommand{\ROM}{\textsf{ROM}\xspace}
\newcommand{\FOM}{\textsf{FOM}\xspace}
\newcommand{\LTI}{\textsf{LTI}\xspace}
\newcommand{\SPD}{\textsf{SPD}\xspace}
\newcommand{\SVD}{\textsf{SVD}\xspace}
\newcommand{\ARE}{\textsf{ARE}\xspace}
\newcommand{\PRARE}{\textsf{PR-ARE}\xspace}
\newcommand{\BRARE}{\textsf{BR-ARE}\xspace}
\newcommand{\ALE}{\textsf{ALE}\xspace}
\newcommand{\QBT}{\textsf{QuadBT}\xspace}
\newcommand{\GQBT}{\textsf{GenQuadBT}\xspace}
\newcommand{\ProjMOR}{\textsf{ProjMOR}\xspace}
\newcommand{\BTr}{\textsf{BT}\xspace}

\newcommand{\BST}{\textsf{BST}\xspace}
\newcommand{\PRBT}{\textsf{PRBT}\xspace}
\newcommand{\BRBT}{\textsf{BRBT}\xspace}

\newcommand{\QBST}{\textsf{QuadBST}\xspace}
\newcommand{\QPRBT}{\textsf{QuadPRBT}\xspace}
\newcommand{\QBRBT}{\textsf{QuadBRBT}\xspace}

\newcommand{\MORLAB}{\textsf{MORLAB}\xspace}
\newcommand{\MATLAB}{\textsf{MATLAB}\xspace}

\begin{document}
\title{Generalizations of data-driven balancing: What to sample for different balancing-based reduced models}

\author[$\ast$]{Sean Reiter}
\affil[$\ast$]{Department of Mathematics, Virginia Tech,
  Blacksburg, VA 24061, USA.\authorcr
  \email{seanr7@vt.edu}, \orcid{0000-0002-7510-1530}}
  
\author[$\star$]{Ion Victor Gosea Gugercin}
\affil[$\star$]{Max Planck Institute for Dynamics of Complex Technical Systems,
  Sandtorstr. 1, 39106 Magdeburg, Germany.\authorcr
  \email{gosea@mpi-magdeburg.mpg.de}, \orcid{0000-0003-3580-4116}}

\author[$\dagger$]{Serkan Gugercin}
\affil[$\dagger$]{Department of Mathematics,
  Division of Computational Modeling and Data Analytics, and
  Virginia Tech National Security Institute,
  Virginia Tech, Blacksburg, VA 24061, USA.\authorcr
  \email{gugercin@vt.edu}, \orcid{0000-0003-4564-5999}
}
  
\shorttitle{Generalizations of data-driven balancing: What to sample for \BTr{}-\ROM{}s}
\shortauthor{S. Reiter, I. V. Gosea, S. Gugercin}
\shortdate{2025-06-23}
  

  
\abstract{%
The quadrature-based balanced truncation (\QBT) framework of~\cite{gosea2022data} is a non-intrusive reformulation of balanced truncation (\BTr{}), a classical projection-based model-order reduction technique for linear systems. \QBT is non-intrusive in the sense that it builds approximate balanced truncation reduced-order models entirely from system response data, e.g., transfer function measurements, without the need to reference an explicit state-space realization of the underlying full-order model. In this work, we generalize the \QBT framework to other types of balanced truncation model reduction. Namely, we show what transfer function data are required to compute data-driven reduced models by balanced stochastic truncation, positive-real balanced truncation, and bounded-real balanced truncation. In each case, these data are evaluations of particular spectral factors associated with the system of interest.
These results lay the theoretical foundation for data-driven reformulations of the aforementioned \BTr{} variants.
Although it is not yet clear how to compute or obtain these spectral factor data in a practical real-world setting, examples using synthetic (numerically evaluated) transfer function data are included to validate the data-based reduced models.
}

\novelty{}

\maketitle

\section{Introduction}
\label{sec:intro}

Model-order reduction (\MOR) refers to the procedure by which one approximates a large-scale dynamical system with a surrogate reduced-order model (\ROM)~\cite{antoulas2005approximation,benner2017model,antoulas2020interpolatory,BenGQetal20a}.
\emph{Balanced truncation} (\BTr) model reduction~\cite{moore1981principal,mullis1976synthesis} and its variants~\cite{desai1984transformation,green1988balanced,opdenacker1988contraction,enns1984model} comprise a well-established class of methods for the approximation of linear time-invariant (\LTI) dynamical systems, which are the focus of this work.
The allure of \BTr methods stems from the fact that they preserve desirable qualitative features of the full-order model (\FOM), e.g., asymptotic stability or passivity, and often provide tractable \emph{a priori} bounds on the relative or absolute $\CH_\infty$ approximation error.

\BTr and its variants are \emph{intrusive} by nature; that is, they require access to an explicit state-space formulation of the approximated system's internal dynamics to compute a \ROM from the full-order system matrices via projection.
In this work, we are interested in \emph{data-driven} approaches for reduced-order modeling.
Data-driven methodologies are \emph{non-intrusive} in the sense that they construct surrogate \ROM{}s entirely from input-output invariants, such as impulse response measurements or transfer function evaluations, without the need to reference a particular realization of the \FOM. 
These data can be obtained experimentally, e.g., by measuring the response of some physical system in a laboratory setting, or synthetically via numerical simulation of a computational model.

The recent contribution~\cite{gosea2022data} introduces a data-driven reformulation of the classical \BTr: \emph{quadrature-based balanced truncation} (\QBT). 
Other data-driven formulations of \BTr have been proposed throughout the years; see, e.g.,~\cite{WP02,Ro05,singler2009proper,benner2010balanced}. Indeed, even in~\cite{moore1981principal} Moore had already motivated a data-driven formulation of \BTr based in the time-domain.
However, the methodologies highlighted above require state trajectories that depend upon a particular system realization.
By contrast, \QBT only requires state-invariant transfer function data.
In this work, we develop a theoretical framework for data-driven (quadrature-based) balancing that generalizes \QBT{} to other types of \BTr{}-\MOR{}.
Specifically, we develop data-driven reformulations of balanced stochastic truncation (\BST)~\cite{desai1984transformation,green1988balanced,green1988relative}, positive-real balanced truncation (\PRBT)~\cite{desai1984transformation}, and bounded-real balanced truncation (\BRBT)~\cite{opdenacker1988contraction}.

\textbf{Main contributions.}
The essential quantities in any bal\-ancing-based model reduction are the \emph{system Gramians}: a pair of symmetric positive (semi-)definite matrices that carry some system-theoretic significance.
In the classical (Lyapunov) \BTr{} due to~\cite{moore1981principal,mullis1976synthesis}, these are the reachability and observability Gramians that uniquely satisfy dual algebraic \emph{Lyapunov} equations (\ALE{}s).
In variants of balanced truncation, the Gramians often satisfy algebraic \emph{Riccati} equations (\ARE{}s).
Once the \emph{relevant} Gramians---that is, the Gramians relevant to the type of \BTr{} being performed---are specified, the balancing-based \ROM{} is entirely specified by the action of the Gramians' square root factors on the full-order system matrices.
Otherwise, the algorithmic computation of a \BTr{}-\ROM{} is identical across \BTr{} variants.
By exploiting this insight, we develop here a theoretical framework for data-driven balancing that generalizes \QBT{}~\cite{gosea2022data}, and shows how to construct various types of (approximate) balancing-based reduced models from different input-output invariant frequency-response data.
Our primary contributions are as follows.
\begin{enumerate}
    \item[1.] Theorem~\ref{thm:gen_quadbt} presents a generalized framework for quad\-rature-based balancing.
    This result is an enhanced version of those in~\cite[Prop.~3.2, Prop.~3.3]{gosea2022data} that applies to \emph{any} pair of Gramians which satisfy a generic pair of \ALE{}s, and can thus be expressed as contour integrals.
    At least each of the \BTr{} variants considered in this paper---namely \BST{}, \PRBT{}, and \BRBT{}---can be expressed in this framework.
    By decomposing the relevant Gramians using appropriately chosen quadrature-based square-root factors, we show how to (approximately) compute the reduced-order quantities in an intrusive \BTr{}-\ROM{} from various transfer function data.
    Unlike \QBT{}, these data are not necessarily evaluations of the underlying full-order system transfer function; rather, they are defined by the pair of Gramians being balanced.
    
    \item[2.] By applying the result of Theorem~\ref{thm:gen_quadbt} to the variants \BST, \PRBT, and \BRBT, 
    Theorems~\ref{thm:bst_from_data}~--~\ref{thm:brbt_from_data} answer the titular question: what do you need to sample for different (data-driven) \BTr reduced models? 
    In each instance, we show that the data required to approximately mimic these types of \BTr{} are evaluations of certain input-output invariant \emph{spectral factors} associated with the original system's transfer function.
\end{enumerate}

In contrast to \QBT{}, it is not yet clear how to measure the response of these spectral factors in a real-world experimental setting.
In this sense, the methodology developed here is not yet practical for a completely data-driven setting.
Nonetheless, the reduced-order models we derive are non-intrusive in the sense that they make no explicit reference to internal quantities once the requisite data are provided.
Thus, the results of this paper lay the theoretical foundation for data-driven implementations of \BST, \PRBT, and \BRBT.

\textbf{Organization.}
The rest of the work is organized as follows.
Section~\ref{sec:bt} reviews the requisite linear systems theory and outlines the key details of Lyapunov \BTr{}~\cite{moore1981principal,mullis1976synthesis} to ground the subsequent discussion.
Section~\ref{sec:gen_quadbt} presents a generalized formulation of \QBT that extends to any balancing-based \MOR{} wherein the relevant pair Gramians can be expressed as solutions of a pair of \ALE{}s, including the \BTr{} variants studied in this paper.
Section~\ref{sec:applied_genqbt} then interprets the results of Section~\ref{sec:gen_quadbt} in the settings of \BST, \PRBT, and \BRBT to obtain data-driven reformulations of these \BTr variants.
Computational experiments using synthetic (numerically evaluated) transfer function data are provided in Section~\ref{sec:numerics} as a proof of concept to validate the data-based reduced models.
Section~\ref{sec:conclusion} concludes the paper and discusses future work.

\section{Linear systems and Lyapunov \BTr{}}
\label{sec:bt}

Throughout this work, we consider \LTI dynamical systems given in state-space form as
\begin{align}
    \label{eq:sys}
    \Sys:
    \left\{
    \begin{array}{rcl}\dot{\Bx}(t)  & = &  \BA@\Bx(t) + 
    \BB@\Bu(t),~~~\Bx(0)=\Bzero_{n\times 1}, \\
    \By(t) &= & \BC@\Bx(t) +\BD\Bu(t),
    \end{array}
    \right.
\end{align}
where $\Bzero_{n_1\times n_2} \in \R^{n_1\times n_2}$ denotes a matrix with all entries equal to zero.
The inputs, state coordinates and outputs of $\Sys$ at time $t\geq 0$ are contained in $\Bu(t)\in\Rm$, $\Bx(t)\in\Rn$ and $\By(t)\in\Rp$. The matrices $\BA\in\R^{n\times n}$, $\BB\in \R^{n\times m}$, $\BC\in \R^{p\times n}$, $\BD\in \R^{p\times m}$ constitute a \emph{state-space realization} of $\Sys$.
For notation, we use $\left(\BA,\BB,\BC,\BD\right)$ to indicate the particular realization of $\Sys$ in~\cref{eq:sys}.
In this work, we assume that $\Sys$ is \emph{asymptotically stable}, i.e., $\BA$ is Hurwitz and thus its eigenvalues lie in the open left-half plane.
Given a realization of $\Sys$, we assume it is \emph{minimal}, i.e., fully reachable and observable~\cite[Lemma~4.42]{antoulas2005approximation}.
The \emph{transfer function} of $\Sys$, defined as
\begin{align}
    \label{eq:tf}
    \BG(s)\coloneqq\BC(s\BI_n - \BA)^{-1}\BB + \BD\in\Cpm,
\end{align}
for the complex variable $s\in\C$
is a matrix-valued rational function analytic in the closed right half-plane, where $\BI_n$ denotes the $n \ n\times n$ identity matrix. The transfer function fully characterizes the input-to-output mapping of $\Sys$ in the frequency domain and is state-space realization invariant. 
The $\CH_\infty$ norm of $\Sys$ is defined as 
$\|\Sys\|_{\CH_\infty}\coloneqq\sup_{\omega\in\R}\sigma_{\max}\left(\BG(\imunit\omega)\right)$, where $\sigma_{\max}(\cdot)$ denotes the maximal singular value of $(\cdot)$ and $\imunit^2=-1$.
The \emph{dual} to $\Sys$ is defined as the \LTI system ($-\BA^{\trans}$, $-\BC^{\trans}$, $\BB^{\trans}$, $\BD^{\trans}$) having the transfer function $\BG(-s)^{\trans}$.

Given a system $\Sys$ as in~\cref{eq:sys}, we seek a \LTI-\ROM{}
\begin{align}
    \label{eq:sysred}
    \Sysred:
    \left\{
    \begin{array}{rcl}\dot{\Bx}_r(t)  & = &  \BA_r@\Bx_r(t) + 
    \BB_r@\Bu(t),~~~\Bx_r(0)=\Bzero_{r\times 1}, \\
    \By_r(t) &= & \BC_r@\Bx_r(t) +\BD_r\Bu(t),
    \end{array}
    \right.
\end{align}
having the reduced-order transfer function 
\begin{align*}
    \BG_r(s)\coloneqq\BC_r(s\BI_r - \BA_r)^{-1}\BB_r + \BD_r\in\Cpm,
\end{align*}
where $\Bx_r(t)\in\Rr$, $\By_r(t)\in\Rp$, $\BA_r\in\Rrr$, $\BB_r\in \Rrm$, $\BC_r\in \Rpr$, and $\BD_r\in \Rpm$ for $r\ll n$.
To be an effective surrogate, $\Sysred$ should (i) accurately reproduce the input-to-output response of $\Sys$, i.e., $\By_r(t)\approx \By(t)$ for any admissible input $\Bu(t)$, and (ii) preserve important qualitative features the \FOM.
\emph{Projection} is at the core of many model reduction algorithms, including balancing-based methods.
Given left and right model reduction subspaces spanned by $\BW_r\in\Rnr$ and $\BV_r\in\Rnr$ such that $\BW_r^{\trans}\BV_r=\BI_r$, the order-$r$ \ROM{} $(\BA_r, \BB_r, \BC_r, \BD_r)$ computed via \emph{Petrov-Galerkin projection} is determined by
\begin{equation}
\label{eq:pg_proj}
\begin{alignedat}{2}
    &\BA_r =
    {{{\BW_r^{\trans}} \BA {\BV_r}}},~~~\BB_r = {{{\BW_r^{\trans}}}\BB},~~~\BC_r = {{\BC{\BV_r}}}.
\end{alignedat}
\end{equation}
In projection-based model-order reduction~(\ProjMOR), it is common to choose $\BD_r=\BD$. Such methods differ in the way they choose $\BV_r$ ad $\BW_r$; see~\cite{antoulas2005approximation,benner2017model,antoulas2020interpolatory} for more details on \ProjMOR.

\emph{Balanced truncation} (\BTr) and its variants are well-established projection-based methods for the model reduction of \LTI dynamical systems. \BTr-\ROM{}s preserve important system-theoretic features of the \FOM and often satisfy tractable bounds on the $\CH_\infty$ approximation error, which in turn bounds the $\CL_2$ norm of the output error $\By-\By_r$~\cite[Ch.~2.1]{antoulas2020interpolatory}.
The \emph{system Gramians} are the key components of any balancing-based model reduction algorithm. 
In classical \BTr, the Gramians are the unique solutions to dual algebraic \emph{Lyapunov equations} (\ALE{}s). 
In many other variants of \BTr, the Gramians are the \emph{minimal} (or \emph{stabilizing}) solutions of algebraic \emph{Riccati equations} (\ARE{}s).
\emph{Balancing} is just the simultaneous diagonalization of two such Gramians.
Once these two matrices are computed, the \FOM is balanced by an appropriate change of coordinate system in which the pertinent Gramians are diagonal and identical.
Order reduction is then accomplished by truncating the least important components of the state space; these are precisely the states associated with the smallest magnitude singular values of the balanced Gramians. 
In practice, the balancing and truncation steps are implemented simultaneously using the so-called \emph{square-root} algorithm~\cite{laub1987,tombs87};~\cite[Sec.~7.3]{antoulas2005approximation}, which fits under \ProjMOR as in~\cref{eq:pg_proj}.

Next in Section~\ref{ss:lyap_bt}, we recount the key details of \BTr{}-\MOR{} in its original Lyapunov setting~\cite{mullis1976synthesis,moore1981principal} to ground the subsequent discussion.
The important details of the other \BTr{} variants considered in this work, as well as the specification of the relevant pair of Gramians in each instance, are self-contained in Section~\ref{sec:applied_genqbt}.
For a more general treatise of \BTr model reduction, see~\cite[Ch.~7]{antoulas2005approximation},~\cite{gugercin2004survey,breiten2021balancing,benner2017modelch6}.

\subsection{Lyapunov \BTr{}}
\label{ss:lyap_bt}

\BTr was independently introduced in the works~\cite{mullis1976synthesis,moore1981principal}. 
In its original setting, which we refer to as \emph{Lyapunov} \BTr{}, the central quantities are the \emph{observability} and \emph{reachability} Gramians $\BQ\in\Rnn$ and $\BP\in\Rnn$ of $\Sys$. These are given as the unique solutions to dual \ALE{}s:
\begin{align}
    \label{eq:obsv_lyap}
    \BA^{\trans}\BQ+ \BQ\BA + \BC^{\trans}\BC&=\Bzero\\
    \label{eq:reach_lyap}
    \mbox{and}~~\BA\BP + \BP\BA^{\trans} + \BB\BB^{\trans}&=\Bzero.
\end{align}
The uniqueness of $\BQ$ and $\BP$ follows from the asymptotic stability of $\Sys$~\cite[Prop.~6.2]{antoulas2005approximation}. 

The Gramians $\BQ$ and $\BP$ are symmetric positive definite (\SPD) matrices by minimality of $\Sys$, and in turn, there exist nonsingular matrices $\BL, \BU\in\Rnn$ such that $\BQ=\BL\BL^{\trans}$ and $\BP=\BU\BU^{\trans}$.
The \emph{Hankel singular values} of $\Sys$ are defined as the singular values of $\BL^{\trans}\BU$, denoted $\sigma\big(\BL^{\trans}\BU\big)$.
In Lyapunov balanced coordinates, i.e., the coordinate system in which $\BQ$ and $\BP$ are diagonal and equal,
states that are weakly reachable are simultaneously weakly observable. 
These are precisely the states identified by the smallest magnitude Hankel singular values. A Lyapunov \BTr-\ROM{} is obtained by effectively truncating those components of the state space. 
Lyapunov \BTr-\ROM{}s retain the asymptotic stability of the \FOM~\cite{pernebo1982model} and satisfy an \emph{a priori} upper bound on the $\CH_\infty$ approximation error $\|\Sys-\Sysred\|_{\CH_\infty}$ in terms of the neglected Hankel singular values~\cite{enns1984model}.

\section{A generalized framework for \QBT}
\label{sec:gen_quadbt}
Traditional formulations of Lyapunov \BTr{} and its variants are \emph{intrusive} insofar as they require a state-space realization $\BA$, $\BB$, and $\BC$ of a linear system $\Sys$ to compute to compute the \BTr{}-\ROM{} by projection according to~\cref{eq:pg_proj}.
By contrast, the recent contribution~\cite{gosea2022data} derives a \emph{non-intrusive}, or data-driven reformulation of \BTr, called \emph{quadrature-based balanced truncation} (\QBT).
Data for our purposes refers to input-to-output frequency-response data, e.g., particular transfer functions evaluated along the imaginary axis $\imunit\R$.
As its name suggests, \QBT{} achieves a non-intrusive formulation by replacing the exact square-root factors of the Gramians $\BQ$ and $\BP$ in the square-root algorithm with approximate quadrature-based factors derived from numerical quadrature rules used to implicitly approximate $\BQ$ and $\BP$.

While the \QBT{} algorithm~\cite{gosea2022data} is non-intrusive, its presentation is limited to the usual Lyapunov setting where the Gramians being balanced are $\BQ$ and $\BP$ that satisfy~\cref{eq:obsv_lyap} and~\cref{eq:reach_lyap}.
In this section, we develop a generalized framework for quadrature-based (data-driven) balancing that is applicable to a wider range of \BTr{} variants, and in particular those mentioned in Section~\ref{sec:intro}: \BST{}, \PRBT{}, and \BRBT{}.
This generalized presentation contains \QBT{} as a special case; see Remark~\ref{remark:QBT_case}.

\subsection{Square-root Lyapunov \BTr for generic Gramians}
\label{ss:sqrt_bt}

As the first step towards extending \QBT{} to other types of balancing, we describe here a framework for Lyapunov \BTr{} wherein the Gramians are the unique solutions to a generic pair of \ALE{}s, not necessarily~\cref{eq:reach_lyap} or~\cref{eq:obsv_lyap}.
The methods of \BST, \PRBT, and \BRBT can all be interpreted from this perspective,
and the generalized formulation of \QBT{} that we derive is based on the Lyapunov formulation of these variants.

Let $\BQgen\in\Rnn$ and $\BPgen\in\Rnn$ denote an arbitrary pair of Gramians relevant to the type of balancing being considered.
For example, in the Lyapunov setting, we would have $\BQgen=\BQ$ and $\BPgen=\BP$ that solve~\cref{eq:obsv_lyap} and~\cref{eq:reach_lyap}.
The key assumption we make is that the matrices $\BQgen$ and $\BPgen$ are the solutions to a pair of \ALE{}s:
\begin{align}
\label{eq:gen_obsv_lyap}
\BA^{\trans}\BQgen + \BQgen\BA + \BCgenY^{\trans}\BCgenY&=\Bzero\\
\label{eq:gen_reach_lyap}
\mbox{and}~~
\BA\BPgen + \BPgen\BA^{\trans} +\BBgenX\BBgenX^{\trans}&=\Bzero,
\end{align}
where $\BBgenX\in\R^{\nx\times \mx}$ and $\BCgenY\in\R^{\py\times \ny}$ for positive integers $\mx,\py$.
Because $\BA$ is assumed Hurwitz, the solutions $\BQgen$ and $\BPgen$ to~\cref{eq:gen_obsv_lyap} and~\cref{eq:gen_reach_lyap} are \emph{unique}.
Equations~\cref{eq:gen_obsv_lyap} and~\cref{eq:gen_reach_lyap} yield the following perspective: the matrices $\BPgen$ and $\BQgen$ can be viewed respectively as the observability and reachability Gramians of a pair of asymptotically stable linear systems $\CX$ and $\CY$ defined according to~\cref{eq:sys}:
\begin{equation}
\label{eq:gen_sys}
    \CY=\left(\BAgenY,\BBgenY,\BCgenY,\BDgenY\right),~~\CX=\left(\BAgenX,\BBgenX,\BCgenX,\BDgenX\right),
\end{equation}
where $\BBgenY\in\R^{\ny \times \my}$, $\BDgenY\in\R^{\py\times \my}$ and $\BCgenX\in\R^{\px\times \nx}$, $\BDgenX\in\R^{\px\times \mx}$ for positive integers $\my,\px$.
For ease of reference later on, we refer to~\cref{eq:gen_obsv_lyap} as the \emph{observability Lyapunov equation} of the system $\CY$, and~\cref{eq:gen_reach_lyap} as the \emph{reachability Lyapunov equation} of the system $\CX$.
We further assume that the given realizations of $\CY$ and $\CX$ are minimal so that $\BQgen$ and $\BPgen$ are \SPD{}.
The Gramians that appear in different \BTr{} variants correspond to different choices of $\CY$ and $\CX$, and more specifically, particular choices of the matrices $\BCgenY$ and $\BBgenX$.
Exact formulations of these systems, along with the corresponding state-space quadruples, in the context of each \BTr{} variant we consider are provided in Section~\ref{sec:applied_genqbt}.

Given the Gramians $\BQgen$ and $\BPgen$, one can reduce the model order of a linear system $\Sys$ by first computing a balancing transformation such that $\BQgen$ and $\BPgen$ are diagonal and equal, and subsequently truncating states corresponding to the smallest singular values of these balanced Gramians. 
(This is the same general procedure underlying Lyapunov \BTr{} described in Section~\ref{ss:lyap_bt}, but with $\BQgen$ and $\BPgen$ taking the place of $\BQ$ and $\BP$.)
In a numerical setting, one never computes the full-balancing transformation since this is notoriously ill-conditioned~\cite[Sec.~7.3]{antoulas2005approximation}. 
Indeed, one does not even need $\BQgen$ and $\BPgen$, but only their \emph{square-root factors}
$\BLgenY\in\Rnn$ and $\BUgenX\in\Rnn$ such that
\begin{equation}
\label{eqn:UxLy}
    \BQgen = \BLgenY\BLgenY^{\trans}~~\mbox{and}~~\BPgen = \BUgenX\BUgenX^{\trans}.
\end{equation}
The factors $\BLgenY$ and $\BUgenX$ can be obtained directly without forming $\BQgen$ and $\BPgen$ explicitly; see, e.g., the surveys \cite{morBenS11,simoncini2016computational}. 
As already highlighted in Section~\ref{sec:bt}, in a practical setting, balancing and truncation are achieved simultaneously using the \emph{square-root algorithm} for \BTr~\cite{laub1987,tombs87};~\cite[Ch~7.4]{antoulas2005approximation}.
The square-root algorithm, which is written down in Algorithm~\ref{alg:sqrt_bt}, falls under \ProjMOR where the left and right projection subspaces spanned by $\BW_r$ and $\BV_r$ are obtained from the singular-value decomposition (\SVD) of $\BLgenY^{\trans}\BUgenX$.
This implementation is numerically well-conditioned and lends itself to a low-rank implementation by replacing the $\BLgenY$ and $\BUgenX$ with approximate \emph{low-rank} factors $\BLgenYTilde$ and $\BUgenXTilde$. 

\begin{algorithm}[h!]
\caption{Square-root \BTr{}\label{alg:sqrt_bt}}
\begin{algorithmic}
\STATE 
\STATE \textbf{Input:} $\BA\in\Rnn,$ $\BB\in\Rnm$, $\BC\in\Rpn,$ and order $r$ $(1\leq r < n)$ so that $\sigma_r > \sigma_{r+1}$.\\[1ex]
\STATE \textbf{Output:} \BTr-\ROM{} given by $\BA_r\in\Rrr,$ $\BB_r\in\Rrm$, $\BC_r\in\Rpr.$\\[1ex]
\begin{enumerate}
    \item Obtain square-root factors $\BUgenX$, $\BLgenY\in\Rnn$ of \underline{\emph{the relevant system Gramians}} $\BPgen$, $\BQgen\in\Rnn$.
    \item Compute the \SVD of $\BLgenY^{\trans}\BUgenX$:
    \begin{align} \label{eqn:svdLTU}
    \BLgenY^{\trans}\BUgenX=
    \begin{bmatrix} 
    \BZ_1 & \BZ_2
    \end{bmatrix}
     \begin{bmatrix} 
    \BSigma_1 & \\
    & \BSigma_2
    \end{bmatrix}
     \begin{bmatrix} 
    \BY_1^{\trans} \\ \BY_2^{\trans}
    \end{bmatrix},
    \end{align}
    where $\BSigma_1\in\Rrr,$ $\BSigma_2\in\R^{(n-r)\times(n-r)},$ $\BZ_1,\BY_1\in\Rnr$ and $\BZ_2,\BY_2\in\R^{n\times(n-r)}.$ 
    \item Build Petrov-Galerkin model reduction matrices:
\begin{equation} \label{eqn:VrWr}
\BW_r=\BLgenY^{{\trans}}\BZ_1\BSigma_1^{-1/2}, ~~ \BV_r=\BUgenX\BY_1\BSigma_1^{-1/2},
\end{equation}
where $\BW_r^{\trans}\BV_r=\BI_r$ by construction.
    \item Compute the \BTr-\ROM{} using $\BW_r$ and $\BV_r$:
    \begin{equation}
    \label{eqn:romBT}
    \begin{aligned}
    \BA_r&=\BSigma_1^{-1/2}\BZ_1^{\trans}\left(\BLgenY^{{\trans}}\BA\BUgenX\right)\BY_1\BSigma_1^{-1/2},  \\
    \BB_r &=\BSigma_1^{-1/2}\BZ_1^{\trans}\left(\BLgenY^{{\trans}}\BB\right)&,
     \\
    \BC_r&=\left(\BC\BUgenX\right)\BY_1\BSigma_1^{-1/2}. 
    \end{aligned}
    \end{equation}
\end{enumerate}
\end{algorithmic}
\end{algorithm}

We emphasize that the square-root implementation of Algorithm~\ref{alg:sqrt_bt} is applicable \emph{any variant of \BTr} for which the relevant Gramians can be cast as solutions to the \ALE{}s~\cref{eq:gen_obsv_lyap} and~\cref{eq:gen_reach_lyap} for an appropriate choice of $\BCgenY$ and $\BBgenX$, and in particular the \BTr{} variants considered in this work. 
Once the square-root factors $\BLgenY$ and $\BUgenX$ of the relevant Gramians are specified, Steps 2--4 of Algorithm~\ref{alg:sqrt_bt} proceed identically.
The same can be said for a \emph{quadrature-based} implementation; one only needs that the relevant Gramians elicit exploitable integral representations so that $\BLgenY$ and $\BUgenX$ may be replaced by appropriately chosen quadrature-based factors. The assumption that $\BQgen$ and $\BPgen$ satisfy the \ALE{}s~\cref{eq:gen_obsv_lyap} and~\cref{eq:gen_reach_lyap} provides exactly this structure.
This insight enables us to derive a generalized formulation of \QBT{} that is applicable to any type of \BTr{}-\MOR{} for which the relevant Gramians can be viewed as solutions to the \ALE{}s~\cref{eq:gen_obsv_lyap} and~\cref{eq:gen_reach_lyap} for an appropriate choice of $\BCgenY$ and $\BBgenX$.

\subsection{Theoretical formulation for Generalized \QBT}
Consider the square-root factors $\BUgenX$ and $\BLgenY$ of $\BPgen$ and $\BQgen$. Recall from~\cref{eqn:VrWr} in Algorithm~\ref{alg:sqrt_bt} that the \BTr-\ProjMOR bases are defined as 
\begin{equation*}
    \BW_r=\BLgenY\BZ_1\BSigma_1^{-1/2}~~\mbox{and}~~ \BV_r=\BUgenX\BY_1\BSigma_1^{-1/2}.
\end{equation*}
Because $\BZ_1$, $\BY_1$, and $\BSigma_1$ are extracted from the \SVD of $\BLgenY^{\trans}\BUgenX$, the (intrusive) \BTr-\ROM given by~\cref{eqn:romBT} is completely specified by the four quantities:
\begin{align}
    \label{eq:key_bt_quantities}
    \BLgenY^{\trans}\BUgenX,\quad \BLgenY^{\trans}\BA\BUgenX,\quad \BLgenY^{\trans}\BB, \quad \mbox{and} \quad \BC\BUgenX.
\end{align}
We derive approximations to the matrices in~\cref{eq:key_bt_quantities} that are constructed entirely from different state-invariant input-output data.
This will be accomplished by replacing $\BUgenX$ and $\BLgenY$ with certain quadrature-based square-root factors derived from numerical quadrature rules used to implicitly approximate $\BPgen$ and $\BQgen$. These numerical quadratures are indeed \emph{never constructed} but form the basis of the analysis.
First, we introduce two definitions to aid our exposition.

\begin{definition}
\label{def:matrix_labeling}
Let $\BM\in\C^{(pK)\times(mJ)}$. Then for $1\leq j\leq J$ and $1\leq k \leq K$, the 
matrix $\BM_{k,j}\in\Cpm$ denotes 
the $(k,j)$-th block of $\BM$.
{When $J=1$ or $K=1$, we use $\BM_i$ to denote the $i$th $p\times m$-sized row or column block of $\BM$, respectively.}
In the SISO case of $m=p=1$, $\BM_{k,j} = \BM(k,j)$ is a scalar quantity.
\end{definition}

\begin{definition}
\label{def:strictly_proper}
    For a proper rational function $$\BG(s)=\BC(s\BI_n-\BA)^{-1}\BB+\BD\in\Cpm$$ as in~\cref{eq:tf}, we denote the \emph{strictly proper} part of $\BG(s)$ by $\BG_{\infty}(s)$, defined as
    \begin{align}
        \label{eq:strictly_proper}
        \BG_{\infty}(s)=\BC(s\BI_n-\BA)^{-1}\BB = \BG(s) - \lim_{s\rightarrow\infty}\BG(s).
    \end{align}
\end{definition}

Recall that the linear systems $\CX$ and $\CY$ in~\cref{eq:gen_sys} are asymptotically stable because they share the same $\BA$ matrix as the underlying model $\Sys$.
Consequently, the solutions to the \ALE{}s~\cref{eq:gen_reach_lyap} and~\cref{eq:gen_obsv_lyap} are unique~\cite[Prop.~6.2]{antoulas2005approximation} and can be expressed as contour integrals
\begin{align}
\label{eq:reach_gram}
 \BPgen = \frac{1}{2\pi}\int_{\imunit\R}(z \BI_n - \BA)^{-1}\BBgenX
    \left((z\BI_n - \BA)^{-1}\BBgenX\right)^{\herm}\,dz,\\
\label{eq:obsvgram}
 \BQgen = \frac{1}{2\pi}\int_{\imunit\R}\left(\BCgenY(z \BI_n - \BA)^{-1}\right)^{\herm}
   \BCgenY(z\BI_n - \BA)^{-1}\,dz,
\end{align}
where $\BX^{\herm}$ denotes the Hermitian transpose of a matrix $\BX\in\Cnn$.
Consider a numerical quadrature rule defined by the weights $\rho_j^2$ and nodes $\zeta_j$ for $j=1,2,\ldots, J.$
Applying this rule to the integral form of $\BPgen$ in~\cref{eq:reach_gram} reveals the approximate \emph{quadrature-based} factorization
\begin{align*}
 \BPgen&\approx   \sum_{j=1}^{J}\rho_j^2(\imunit\zeta_j \BI_n - \BA)\BBgenX\left((\imunit\zeta_j \BI_n - \BA)\BBgenX\right)^{\herm}\\
 &=\BUgenXTilde\BUgenXTilde^{\herm},
\end{align*}
where $\BUgenXTilde\in\C^{n\times (m_xJ)}$ is defined by
\begin{align}
    \label{eq:gen_cholU}
  \big(\BUgenXTilde\big)_j= \rho_j\left(\imunit\zeta_j\BI_n - \BA\right)^{-1}\BBgenX\in\C^{n\times m_x},
\end{align}
for all $j=1,\ldots, J$ according to the notation in Definition~\ref{def:matrix_labeling}.
A similar quadrature-based factorization can be obtained for $\BQgen$ in~\cref{eq:obsvgram}. Consider a numerical quadrature rule defined by the weights $\phi_k^2$ and nodes $\omega_k$ for all $k=1,2,\ldots, K$. 
Then, $\BQgen\approx\BLgenYTilde\BLgenYTilde^{\herm}$ where $\BLgenYTilde^{\herm}\in\C^{(p_yK)\times n}$ is defined by
\begin{align}
    \label{eq:gen_cholL}
    \big(\BLgenYTilde^{\herm}\big)_k=
    \phi_k\BCgenY(\imunit\omega_k \BI_n -\BA)^{-1}\in\C^{p_y\times n},
\end{align}
for all $k=1,\ldots,K$.
Replacing the exact factors $\BUgenX$ and $\BLgenY$ in Algorithm~\ref{alg:sqrt_bt} with $\BUgenXTilde$ and $\BLgenYTilde$ already yields a low-rank implementation of \BTr{}.
The resulting (approximate) \BTr{}-\ROM{} is completely determined by the quadrature-based quantities
\begin{align}
    \label{eq:key_gen_quadbt_quantities}
    \BLgenYTilde^{\herm}\BUgenXTilde,\quad \BLgenYTilde^{\herm}\BA\BUgenXTilde,\quad \BLgenYTilde^{\herm}\BB, \quad \mbox{and} \quad \BC\BUgenXTilde.
\end{align}
Significantly, the quadrature-based approximations~\cref{eq:key_gen_quadbt_quantities} to the (intrusive) quantities in~\cref{eq:key_bt_quantities} can be computed entirely from \emph{transfer function data}. We prove this next.

\begin{theorem}
    \label{thm:gen_quadbt}
    Define the transfer functions:
    \begin{align}
        \label{eq:GforEA}
        \BGforEA(s)&:=\BCgenY(s\BI_n-\BA)^{-1}\BBgenX\in\C^{p_y\times m_{x}},\\ 
        \BGforB(s) &:=\BCgenY(s\BI_n-\BA)^{-1}\BB\in\C^{p_y\times m},   \label{eq:GforB}\\
         \label{eq:GforC}
        \BGforC(s) &:= \BC(s\BI_n-\BA)^{-1}\BBgenX\in\C^{p\times m_{x}}.
    \end{align}
    Let $\BUgenXTilde$ and $\BLgenYTilde$ be defined as in~\cref{eq:gen_cholU} and~\cref{eq:gen_cholL} respectively. Define the matrices
    {\begin{align}
    \begin{split}
        \label{eq:gen_loewner_quantities}
        {\bbL} &:= \BLgenYTilde^{\herm}\BUgenXTilde\in\C^{(p_yK)\times(m_xJ)},\\
        {\bbM} &:= \BLgenYTilde^{\herm}\BA\BUgenXTilde\in\C^{(p_yK)\times(m_xJ)},\\
        {\bbH} &:= \BLgenYTilde^{\herm}\BB\in\C^{(p_yK)\times m},\\
        {\bbG} &:= \BC\BUgenXTilde\in\C^{p\times (m_xJ)}.
    \end{split}
    \end{align}}
    Then, the $(k,j)$-th blocks of $\bbL$ and $\bbM$ are given by
    \begin{align}
        &\bbL_{k,j} = -\phi_k\rho_j\frac{\BGforEA(\imunit\omega_k)-\BGforEA(\imunit\zeta_j)}{\imunit\omega_k-\imunit\zeta_j}, \label{eqn:Lkj}\\
        &\bbM_{k,j} = -\phi_k\rho_j\frac{\imunit\omega_k\BGforEA(\imunit\omega_k)-\imunit\zeta_j\BGforEA(\imunit\zeta_j)}{\imunit\omega_k-\imunit\zeta_j}, \label{eqn:Mkj}
    \end{align}
    and the $k$-th and $j$-th blocks of $\bbH$ and $\bbG$ are given by
    \begin{align}
        \bbH_k=\rho_k\BGforB(\imunit\omega_k) \quad\mbox{and}\quad        \bbG_j=\phi_j\BGforC(\imunit\zeta_j), \label{eqn:HkGj}
    \end{align}
    for each $1\leq k \leq K$ and $1\leq j \leq J$.
\end{theorem}

\begin{proof}{Proof of Theorem~\ref{thm:gen_quadbt}.}
   Let $\Be_i$ be the $i$th canonical unit vector, i.e., its $i$th entry is $1$, and all other entries are $0$. 
   Additionally, for any $i, \ell\in\mathbb{N}$, define the matrix
    \[\BE_{i,\ell}=\begin{bmatrix}
        \Be_{(i-1)\ell + 1} & \Be_{(i-1)\ell + 2} & \cdots & \Be_{i\ell}
    \end{bmatrix}\in\R^{n\times \ell}.\]
    The proof exploits two resolvent identities.
    The first is
    \small
    \begin{align} 
        \label{eqn:res1}
       \hspace{-2mm} (s&\BI_n-\BA)^{-1}(z\BI_n-\BA)^{-1}
        =\frac{(z\BI_n-\BA)^{-1}-(s\BI_n-\BA)^{-1}}{s-z}
    \end{align}
    \normalsize
    for all $s,z\in\C$ that are not in the spectrum of $\BA$.
    Using the definitions of \ $\bbL$ in~\cref{eq:gen_loewner_quantities}, $\BUgenXTilde$ in~\cref{eq:gen_cholU}, $\wtL_{\CY}$ in~\cref{eq:gen_cholL}, and the identity~\cref{eqn:res1}, it follows that 
    \small
    \begin{align*}
        \bbL_{k, j} &= {\BE_{k,p_y}^{\trans}} \bbL{\BE_{j,m_x}} = \left({\BE_{k,p_y}^{\trans}}\BLgenYTilde^{\herm}\right)\left(\BUgenXTilde{\BE_{j,m_x}}\right)\\
        &=\phi_k\rho_j\BCgenY(\imunit\omega_k\BI_n-\BA)^{-1}(\imunit\zeta_j\BI_n-\BA)^{-1}\BBgenX\\
            &= \phi_k \rho_j\BCgenY\left(\frac{(\imunit\zeta_j\BI_n-\BA)^{-1}-(\imunit\omega_k\BI_n-\BA)^{-1}}{\imunit\omega_k - \imunit\zeta_j}\right)\BBgenX\\
            &=-\phi_k\rho_j\frac{\BGforEA(\imunit\omega_k)-\BGforEA(\imunit\zeta_j)}{\imunit\omega_k-\imunit\zeta_j},
    \end{align*}
    \normalsize
    where the last line follows from the definition of $\BGforEA(s)$ in~\cref{eq:GforEA}, thus proving~\cref{eqn:Lkj}. The second resolvent identity we exploit is 
    \begin{align}  
    \begin{split}
    \label{eqn:res2}
        z(z& \BI_n-\BA)^{-1} - s(s\BI_n-\BA)^{-1} \\  
        &= -(z-s)(z\BI_n-\BA)^{-1} \BA(s\BI_n-\BA)^{-1} 
    \end{split}
    \end{align}
    for all $s,z\in\C$ that are not in the spectrum of $\BA$.
    To prove~\cref{eqn:Mkj}, we use~\cref{eqn:res2} and the definitions of $\bbM$ in~\cref{eq:gen_loewner_quantities}, $\BUgenXTilde$ in~\cref{eq:gen_cholU}, and $\BLgenYTilde$ in~\cref{eq:gen_cholL} to obtain
    \small
    \begin{align*}
        \bbM_{k, j} &= {\BE_{k,p_y}^{\trans}} \bbM{\BE_{j,m_x}} = \left({\BE_{k,p_y}^{\trans}}\BLgenYTilde^{\herm}\right)\BA\left(\BUgenXTilde{\BE_{j,m_x}}\right)\\
        &=\phi_k\rho_j\BCgenY(\imunit\omega_k\BI_n-\BA)^{-1}\BA(\imunit\zeta_j\BI_n-\BA)^{-1}\BBgenX\\
         & \hspace{-3ex} = \phi_k \rho_j\BCgenY\left(\frac{\imunit\zeta_j(\imunit\zeta_j\BI_n-\BA)^{-1}-\imunit\omega_k(\imunit\omega_k\BI_n-\BA)^{-1}}{\imunit\omega_k - \imunit\zeta_j}\right)\BBgenX\\
        &=-\phi_k\rho_j\frac{\imunit\omega_k\BGforEA(\imunit\omega_k)-\imunit\zeta_j\BGforEA(\imunit\zeta_j)}{\imunit\omega_k-\imunit\zeta_j}.
    \end{align*}
    \normalsize
    The final claims~\cref{eqn:HkGj} for $\bbH$ and $\bbG$ follow directly from the definitions of $\bbH$ and $\wtL_{\CY}$, as well as $\bbG$ and $\BUgenXTilde$. Observe that
        \begin{align*}
        \bbH_k = {\BE_{k,p_y}^{\trans}} \bbH 
        =\phi_k\BCgenY\left(\imunit\omega_k\BI_n-\BA\right)^{-1}\BB=\phi_k\BGforB(\imunit\omega_k). 
    \end{align*}
    Likewise, for $\bbG$:
    \begin{align*}
        \bbG_j = \bbG{\BE_{j,m_x}} 
        =\rho_j\BC\left(\imunit\zeta_j\BI_n-\BA\right)^{-1}\BBgenX=\rho_j\BGforC(\imunit\zeta_j),
    \end{align*}
    thus completing the proof.
\end{proof}

By replacing the exact quantities in~\cref{eq:key_bt_quantities} with the (approximate) quadrature-based quantities~\cref{eq:key_gen_quadbt_quantities} that are computed from data, we obtain a data-driven formulation of Algorithm~\ref{alg:sqrt_bt}.
We refer to this data-driven formulation as \emph{generalized quadrature-based balanced truncation} (\GQBT).
The choice of notation $\BGforEA(s)$, $\BGforB(s)$, and $\BGforC(s)$ for the transfer functions in~\cref{eq:GforEA}-\cref{eq:GforC} is intentional: the \emph{underscored quantities} in each transfer function are the \emph{reduced-order quantities} in the data-driven surrogate model that require samples of that transfer function.
In other words, Theorem~\ref{thm:gen_quadbt} and the associated notation can be interpreted as follows: 
\begin{enumerate}
    \item[1.] The construction of $\bbL$, and hence its \SVD, and the reduced-order $\wtBA_r$ in Steps (2) and (3) of Algorithm~\ref{alg:gen_quadbt} require samples of $\BGforEA(s)$;
    \item[2.] Construction of the reduced-order $\wtBB_r$ in Step (3) of Algorithm~\ref{alg:gen_quadbt} requires samples of $\BGforB(s)$; 
    \item[3.] Construction of the reduced-order $\wtBC_r$ in Step (3) of Algorithm~\ref{alg:gen_quadbt} requires samples of $\BGforC(s)$.
\end{enumerate}

\subsection{Algorithmic formulation for Generalized \QBT}

Having established its theoretical foundation,
an algorithmic formulation for \GQBT is presented
in Algorithm~\ref{alg:gen_quadbt}.
In principle, Algorithm~\ref{alg:gen_quadbt} only requires the left and right quadrature weights/nodes used to implicitly approximate $\BPgen$ and $\BQgen$, and samples of the transfer functions $\BGforEA(s)$, $\BGforB(s)$, and $\BGforC(s)$ given in~\cref{eq:GforEA}---\cref{eq:GforC} evaluated at these nodes (or at least, the ability to evaluate them).

\begin{algorithm}[hh]
\caption{Generalized-\QBT (\GQBT)\label{alg:gen_quadbt}}
\begin{algorithmic}
\STATE 
\STATE \textbf{Input:} Mappings $\BGforEA(s)$, $\BGforB(s)$, $\BGforC(s)$, order $r$ $(1\leq r <n)$, and:
\begin{itemize}
    \item ``Left'' weights/nodes $\{\rho_j,\zeta_j\}_{j=1}^J$,
    \item ``Right'' weights/nodes  $\{\phi_k,\omega_k\}_{k=1}^K$.
\end{itemize}

\STATE \textbf{Output:} \GQBT-\ROM determined by the state-space matrices $\wtBA_r\in\Crr,$ $\wtBB_r\in\Crm$, $\wtBC_r\in\Cpr.$
\begin{enumerate}
    \item Obtain samples $\{\BGforEA(\imunit\zeta_j)\}_{j=1}^J$, $\{\BGforC(\imunit\zeta_j)\}_{j=1}^J$, $\{\BGforEA(\imunit\omega_k)\}_{k=1}^K$, and $\{\BGforB(\imunit\omega_k)\}_{k=1}^K$. 
    Construct matrices $\big(\bbL, \bbM, \bbH, \bbG\big)$ according to Theorem~\ref{thm:gen_quadbt}.
    \item Compute the \SVD of $\bbL$:
    \begin{align*}
    \bbL=
    \begin{bmatrix} 
    \wtBZ_1 & \wtBZ_2
    \end{bmatrix}
     \begin{bmatrix} 
    \wtBSigma_1 & \\
    & \wtBSigma_2
    \end{bmatrix}
     \begin{bmatrix} 
    \wtBY_1^{\herm} \\ \wtBY_2^{\herm}
    \end{bmatrix}\in\C^{(p_yK)\times (m_xJ)},
    \end{align*}
    where $\wtBSigma_1\in\Rrr,$ $\wtBSigma_2\in\R^{(p_yK-r)\times(m_xJ-r)},$ $\wtBZ_1,\wtBY_1$ and $\wtBZ_2,\wtBY_2$ are partitioned conformally. 
    \item Compute the \GQBT-\ROM:
    \begin{align*}
    \wtBA_r&=\wtBSigma_1^{-1/2}\wtBZ_1^{\herm}\big(\bbM\big)\wtBY_1\wtBSigma_1^{-1/2}\\
    \wtBB_r &=\wtBSigma_1^{-1/2}\wtBZ_1^{\herm}\big(\bbH\big),\\
    \wtBC_r&= \big(\bbG\big)\wtBY_1\wtBSigma_1^{-1/2}.
    \end{align*}
\end{enumerate}
\end{algorithmic}
\end{algorithm}

Some remarks are in order.
We emphasize that at no point do we explicitly compute the quadrature-based approximations of the Gramians $\BPgen$ and $\BQgen$.
These quadrature rules are leveraged \emph{implicitly} to derive the quadrature-based square-root factors in~\cref{eq:gen_cholU} and~\cref{eq:gen_cholL}, and subsequently realize the approximate \BTr-\ROM{} from input-output data.
The error bounds from~\cite[Prop.~3.2]{gosea2022data} that relate the quadrature error to the error in the reduced model matrices apply in this generalized setting, as well.
At face value, the result of Theorem~\ref{thm:gen_quadbt} is not fundamentally different from the main results of~\cite{gosea2022data}. 
Instead, the power of Theorem~\ref{thm:gen_quadbt} lies in its generality; multiple \BTr{} variants wherein the Gramians are solutions to \ARE{}s can be framed in the Lyapunov setting of Section~\ref{ss:sqrt_bt}, allowing for the application of Theorem~\ref{thm:gen_quadbt}. 
Because of this, Theorem~\ref{thm:gen_quadbt} provides the foundational theoretical tool for deriving the data-driven formulations of \BST{}, \PRBT{}, and \BRBT{} that are presented in Section~\ref{sec:applied_genqbt}.
The key deviation of this work from~\cite{gosea2022data} is that the transfer function evaluations required for \GQBT{} are not necessarily those of $\BG(s)$, the transfer function of the linear model being approximated. Rather, Algorithm~\ref{alg:gen_quadbt} requires samples of $\BGforEA(s)$, $\BGforB(s)$, and $\BGforC(s)$ as in~\cref{eq:GforEA}---\cref{eq:GforC}. 
Lastly, we highlight that Algorithm~\ref{alg:gen_quadbt} avoids any explicit reference to internal quantities, e.g., a state-space realization of $\Sys$, or any other linear model.

\begin{remark}
    \label{remark:QBT_case} Theorem~\ref{thm:gen_quadbt} and Algorithm~\ref{alg:gen_quadbt} contain \QBT~\cite{gosea2022data} as a special case. 
    In the usual Lyapunov setting described in Section~\ref{ss:lyap_bt}, we simply have that $\BQgen=\BQ$ and $\BPgen=\BP$, and so the generalized equations~\cref{eq:gen_obsv_lyap} and~\cref{eq:gen_reach_lyap} are the dual \ALE{}s~\cref{eq:obsv_lyap} and~\cref{eq:reach_lyap} corresponding to $\Sys$. Then, $\BBgenX=\BB$, $\BC_\CY=\BC$, and the transfer functions $\BGforEA(s)$, $\BGforB(s)$, and $\BGforC(s)$ all equal $\BG_\infty(s)$.
    This is precisely the aggregate result of~\cite[Prop.~3.1,~Prop.~3.3]{gosea2022data}.
\end{remark}

\section{What to sample for \BTr-variants} 
\label{sec:applied_genqbt}
What remains to be seen is 
what the transfer functions $\BGforEA(s)$, $\BGforB(s)$, and $\BGforC(s)$ in~\cref{eq:GforEA}---\cref{eq:GforC} correspond to for different \BTr{} variants.
We investigate exactly this question next for the cases of \BST, \PRBT, and \BRBT. 
Applying the generalized result of Theorem~\ref{thm:gen_quadbt} to each of the specified variants, we answer the titular question: What does one need to sample for different data-driven \BTr reduced models?
Unlike in the Lyapunov setting, the to-be-sampled data are not necessarily measurements of $\BG(s)$, the transfer function of the \FOM.

In the rest of this section, we sequentially derive data-based formulations of \BST{}, \PRBT{}, and \BRBT{} according to the following structure.
\begin{enumerate}
    \item[1.] We review the key details of these methods and introduce the relevant Gramians, as well as the spectral factorizations that arise from the matrix equations that determine said Gramians.
    \item[2.] We then describe how each variant fits the \GQBT{} framework of Section~\ref{sec:gen_quadbt}. In particular, we provide specific formulations for the linear systems $\CX$ and $\CY$ in~\cref{eq:gen_sys}.
    This enables us to apply Theorem~\ref{thm:gen_quadbt} to each type of \BTr{} for appropriately chosen $\BBgenX$ and $\BCgenY$.
    \item[3.] Finally, we interpret the result of Theorem~\ref{thm:gen_quadbt} applied to the type of \BTr{} in question to show that the transfer functions~\cref{eq:GforEA}---\cref{eq:GforC} can be expressed in terms of the previously introduced spectral factors.
\end{enumerate}
Our treatment of spectral factorizations follows~\cite[Chapter~13.4]{zhou1996robust}; we refer the reader there for a more detailed study.
As we will see, it is not clear how to compute the required spectral factor data in a truly data-driven setup. Thus, the results of this section are intended to establish a complete \emph{theoretical} framework for data-driven \BST{}, \PRBT{}, and \BRBT{}. 
We leave the questions of inferring these spectral factor data from, e.g., evaluations of $\BG(s)$, and developing a computationally practical framework for data-driven \BST{}, \PRBT{}, and \BRBT{} to future research work.

For the results of this section, we use the notation $\left[~~\cdot~~\right]_+$ to denote the \emph{purely stable part} of a rational transfer function. Additionally, when the notation $(\BA, \BB, \BC, \BD)$ does not fit in a single line, we represent the corresponding system by
\begin{equation*}
    \Sys=\left(\begin{array}{ccc|ccc}
         & \BA & & & \BB & \\\hline
         & \BC & & & \BD &
    \end{array}\right).
\end{equation*}

\subsection{\BST from data: \QBST}
\label{ss:qbst}
Model reduction by balanced stochastic truncation (\BST{}) was first introduced in~\cite{desai1984transformation} for the model reduction of stochastic processes, and studied further in~\cite{green1988balanced,green1988relative}.
Suppose that the system $\Sys$ in~\cref{eq:sys} has two additional properties: (i) $\Sys$ is \emph{square}, that is, the input and output dimensions are such that $m=p$, and (ii) the input feed-through term $\BD$ is nonsingular.
A system $\Sys$ is said to be \emph{minimum phase} if the poles and the zeros of its transfer function lie in the open left-half plane.
Compared to Lyapunov \BTr{}, \BST provides an \emph{a priori} upper bound on the \emph{relative} $\CH_\infty$ error $\|\Sys^{-1}(\Sys-\Sysred)\|_{\CH_\infty}$~\cite{green1988relative}, and preserves asymptotic stability as well as the minimum phase property~\cite{green1988balanced}.
For an extension to non-square dynamical systems, see~\cite{benner2001efficient}.

In \BST{}, the relevant Gramians are the reachability Gramian $\BP$ of $\Sys$ that solves~\cref{eq:reach_lyap}, and the \emph{minimal} (stabilizing) solution $\BQbst^{-}\in\Rnn$ to the \ARE{}
\begin{align} 
\begin{split}
    \label{eq:bst_riccati}
        \BA^{\trans}&\BQbst + \BQbst\BA  +\\
        &\left(\BC-\BB_\CW^{{\trans}}\BQbst\right)^{{\trans}}\left(\BD\BD^{{\trans}}\right)^{-1}\left(\BC-\BB_\CW^{{\trans}}\BQbst\right)=\Bzero,
\end{split}
\end{align}
where $\BB_\CW\coloneqq\BP\BC^{\trans} + \BB\BD^{\trans}\in\Rnm.$ 
Any solution $\BQbst\in\Rnn$ to~\cref{eq:bst_riccati} is \emph{not} unique.
The \emph{minimal} solution to~\cref{eq:bst_riccati} is the unique \SPD{} matrix $\BQbst^{-}$ that obeys the partial order $0\prec\BQbst^{-}\preceq\BQbst$ for all symmetric solutions $\BQbst$ of~\cref{eq:bst_riccati}; see~\cite[Theorem~13.11]{zhou1996robust}.
In \BST{}, $\BQbst^{-}$ takes the place of the usual observability Gramian.
In order to derive a data-driven formulation of \BST{}, we show how it fits in the \GQBT{} framework of Section~\ref{sec:gen_quadbt}.
Specifically, we interpret $\BP$ and $\BQbst^{-}$ as the reachability and observability Gramians of some appropriately chosen linear systems $\CX$ and $\CY$.
Obviously, $\BPgen=\BP$ is the reachability Gramian of the system being approximated, and solves~\cref{eq:gen_reach_lyap} for $\CX=\Sys$ with $\BBgenX=\BB$.
To interpret $\BQbst^{-}$, we introduce the following spectral factorization.
By~\cite[Corollary~13.28]{zhou1996robust} there exists a \emph{minimal phase} \emph{right spectral factor} $\BW(s)\in\C^{m\times m}$ of $\BG(s)\BG(-s)^{\trans}$, i.e.,
$\BG(s)\BG(-s)^{\trans} = \BW(-s)^{\trans}\BW(s)$ for all $s\in\C$.
Moreover, $\BW(s)$ is the rational transfer function of a minimal, asymptotically stable linear system 
\begin{equation}
\label{eq:bst_spectral_fact}
    \CW=\left(\BA, \BB_\CW, \BC_\CW, \big(\BD\BD^{\trans}\big)^{1/2}\right),
\end{equation}
where $\BB_\CW=\BP\BC^{\trans} + \BB\BD^{\trans}$ as before and $\BC_\CW\in\R^{m\times n}$ is defined as
\begin{align} 
    \begin{split}
    \label{eqn:BwCw}
    \BC_{\CW}=\left(\BD\BD^{\trans}\right)^{1/2}\left(\BC-\BB_{\CW}^{\trans}\BQbst^{-}\right)\in\R^{m\times n}.
    \end{split}
\end{align}
Explicitly, $\BW(s)$ is written as
\begin{equation}
\label{eq:bst_spectral_fact_tf}
    \BW(s)=\BC_\CW\left(s\BI_n-\BA\right)^{-1}\BB_\CW+ \big(\BD\BD^{\trans}\big)^{1/2}.
\end{equation}
We call $\BW(s)$ minimum phase because the solution $\BQbst^{-}$ to~\cref{eq:bst_riccati} used in its construction is the minimal solution of the \ARE{}.
Fixing the right-hand side of the \ARE{}~\cref{eq:bst_riccati} with $\BQbst=\BQbst^{-}$ yields
\begin{align}
    \nonumber
    \BA^{{\trans}}\BQbst + \BQbst\BA \hphantom{\left(\BD\BD^{{\trans}}\right)^{-1}\left(\BC-\BB_\CW^{{\trans}}\BQbst^{-}\right)}~~~~& \\
    \nonumber
        +\underbrace{\left(\BC-\BB_\CW^{{\trans}}\BQbst^{-}\right)^{{\trans}}\left(\BD\BD^{{\trans}}\right)^{-1}\left(\BC-\BB_\CW^{{\trans}}\BQbst^{-}\right)}_{\textstyle
    \begin{array}{c}=\BC_\CW\BC_\CW^{\trans}\end{array}}&=\Bzero,\\
    \label{eq:gen_obsv_lyap_bst}
    \implies~~\BA^{{\trans}}\BQbst+ \BQbst\BA +
        \BC_{\CW}^{\trans}\BC_{\CW}&=\Bzero,
\end{align}
which is an \ALE{} that fits the structure of~\cref{eq:gen_obsv_lyap} for the choice of $\BCgenY=\BC_\CW$ in~\cref{eqn:BwCw}. Moreover,~\cref{eq:gen_obsv_lyap_bst} is uniquely solved by $\BQbst^{-}$.
Thus,~\cref{eq:gen_obsv_lyap_bst} is the \emph{observability Lyapunov equation} of $\CW$ defined in~\cref{eq:bst_spectral_fact}---the linear system associated with the right spectral factor $\BW(s)$---and $\BQgenY=\BQbst^{-}$ is the \emph{observability Gramian} of $\CY=\CW$.

At this point, Theorem~\ref{thm:gen_quadbt} and the \GQBT{} framework of Section~\ref{sec:gen_quadbt} can be applied to derive a data-driven formulation of \BST{}.
We describe the application explicitly in this instance because it is our first time doing so.
Because the relevant Gramians uniquely satisfy \ALE{}s~\cref{eq:reach_lyap} and~\cref{eq:gen_obsv_lyap_bst}, $\BP$ and $\BQbst^{-}$ have integral representations~\cref{eq:reach_gram} and~\cref{eq:obsvgram}.
In particular, $\BQbst^{-}$ can be written as
\begin{align*}
    \BQbst^{-} = \frac{1}{2\pi}&\int_{\imunit \R}\left(\BD^{-1}\left(\BC-\BB_{\CW}^{\trans}\BQbst^{-}\right)
    \left(z\BI_n - \BA\right)^{-1}\right)^{\herm}\times\\
    &\BD^{-1}\left(\BC-\BB_{\CW}^{\trans}\BQbst^{-}\right)
    \left(z\BI_n - \BA\right)^{-1}\,dz.
\end{align*}
From this representation, 
$\BQbst^{-}$ is decomposed via the quadrature-based square-root factor in~\cref{eq:gen_cholL}, i.e.,
\[\big(\wtL_\CW^{\herm}\big)_k=
    \phi_k \BD^{-1}\left(\BC-\BB_{\CW}^{\trans}\BQbst^{-}\right)(\imunit\omega_k \BI_n -\BA)^{-1},\]
for $k=1,\ldots K$.
The quadrature-based factor of $\BP$ is derived similarly, and given by~\cref{eq:gen_cholU} for $\BBgenX=\BB$.
Using these quadrature-based factorizations to define the data matrices in~\cref{eq:gen_loewner_quantities} and Theorem~\ref{thm:gen_quadbt} yields a data-driven formulation of \BST{}.
Theorem~\ref{thm:bst_from_data} shows how to interpret the transfer functions $\BGforEA(s)$, $\BGforB(s)$, and $\BGforC(s)$ in~\cref{eq:GforEA}---\cref{eq:GforC} for this data-driven \BST{} in terms of $\BG(s)$ and the spectral factor $\BW(s)$.

\begin{theorem}
    \label{thm:bst_from_data}
    Let $\BQbst^{-}\in\Rnn$ be the minimal solution to~\cref{eq:bst_riccati} and $\BP\in\Rnn$ be the reachability Gramian of $\Sys$ that solves~\cref{eq:reach_lyap}.
    Then, the transfer functions
    $\BGforEA(s),$ $\BGforB(s)$, and $\BGforC(s)$ defined as in~\cref{eq:GforEA}---\cref{eq:GforC} of Theorem~\ref{thm:gen_quadbt} and Algorithm~\ref{alg:gen_quadbt} are
    \begin{align}  
        \label{eq:GforB_C_bst}
        \BGforC(s) &= {\BG_\infty(s)},~\mbox{and}  \\
        \label{eq:GforA_E_bst}
    \BGforEA(s) &=\BGforB(s) = \left[\left(\BW(-s)^{\trans}\right)^{-1}\BG_\infty(s)\right]_+,
    \end{align}
    where $\BG(s)$ and $\BW(s)$ are defined as in~\cref{eq:tf} and~\cref{eq:bst_spectral_fact_tf}.
\end{theorem}

\begin{proof}{Proof of Theorem~\ref{thm:bst_from_data}.}
    For \BST, $\BBgenX = \BB$ and $\BCgenY=\BC_\CW$ in~\cref{eqn:BwCw}.
    Thus, by definition of $\BGforC(s)$ in~\cref{eq:GforC}
    \begin{equation*}
        \BGforC(s)=\BC(s\BI_n-\BA)^{-1}\BBgenX=\BC(s\BI_n-\BA)^{-1}\BB=\BG_\infty(s),
    \end{equation*}
    proving~\cref{eq:GforB_C_bst}.
    To prove~\cref{eq:GforA_E_bst}, first note that by definition of~\cref{eq:GforEA} and~\cref{eq:GforB}
    \begin{equation*}
        \BGforEA(s) = \BGforB(s)=\BC_\CW(s\BI_n-\BA)^{-1}\BB,
    \end{equation*}
    because $\BCgenY=\BC_\CW$. Thus, showing~\cref{eq:GforA_E_bst} amounts to proving that
    \begin{equation*}
        \BC_\CW(s\BI_n-\BA)^{-1}\BB    =\left[\left(\BW(-s)^{\trans}\right)^{-1}\BG_\infty(s)\right]_+.    
    \end{equation*}
    From the state-space realizations of $\CW$ in~\cref{eq:bst_spectral_fact} and $\Sys$
    in~\cref{eq:sys}, we calculate a realization of the cascaded system with the transfer function $\left(\BW(-s)^{\trans}\right)^{-1}\BG_\infty(s)$:
    \begin{align}  \label{eq:WG}
    \left(\begin{array}{ccccc|ccc}
         & -\BA^{\trans} + \BC_\CW^{\trans}\BD^{-1}\BB_\CW^{\trans} &  & \BC_\CW^{\trans}\BD^{-1}\BC & & & \Bzero_{n\times m} &\\ 
        & \Bzero_{n\times n} &  & \BA & & & \BB &\\ 
         \hline
        & \BD^{-1}\BB_\CW^{\trans} &  &\BD^{-1}\BC & & & \Bzero_{m\times m} &
    \end{array}\right),\end{align}
    see, e.g.,~\cite[Sec.~3.6]{zhou1996robust}.
    The \ARE in~\cref{eq:bst_riccati} can be rearranged into the form
    \begin{align*} 
        &\left(-\BA^{\trans} + \BC_\CW^{\trans}\BD^{-1}\BB_\CW^{\trans}\right)\left(-\BQbst^{-}\right) \\
        &~~~+ \BA\BQbst^{-} 
            + \BC_\CW^{\trans}\BD^{-1}\BC=\Bzero.
    \end{align*}
    Using this reformulation, it becomes clear that the state-space transformation 
    \begin{equation*}
        \BT = \begin{bmatrix}
        \hphantom{-}\BI_n\hphantom{-} & -\BQbst^{-}\hphantom{-}\\
        \Bzero_n & \BI_n
    \end{bmatrix}\in\R^{2n\times 2n}
    \end{equation*}
    decouples the cascaded system realization~\cref{eq:WG}. In other words, the transformed realization satisfies
    \begin{align} \label{eq:WGdiag}
    \left(\begin{array}{ccccc|ccc}
         & -\BA^{\trans} + \BC_\CW^{\trans}\BD^{-1}\BB_\CW^{\trans} &  & \Bzero_{n \times n} & & &\BQbst^{-}\BB&\\ 
        & \Bzero_{n\times n} &  & \BA & & & \BB &\\ 
         \hline
        & \BD^{-1}\BB_\CW^{\trans} &  & \BC_\CW & & & \Bzero_{m\times m} &
    \end{array}\right).\end{align}
    Note that the $(1,1)$ block of the cascaded system is purely antistable, i.e., its poles lie in the open right half-plane, and the (2,2) block is purely stable.
    Because $\left(\BW(-s)^{\trans}\right)^{-1}\BG_\infty(s)$ is the transfer function of~\cref{eq:WGdiag}, it can be written as
    \begin{align*}
         \left(\BW(-s)^{\trans}\right)^{-1}\BG_\infty(s) &= \ \star \ + \  \underbrace{\BC_\CW(s\BI_n-\BA)^{-1}\BB}_{=\BGforEA(s)=\BGforB(s)} \hot{,}
    \end{align*}
   where the unspecified $\star$ corresponds to the \emph{anti-stable} part of  $\left(\BW(-s)^{\trans}\right)^{-1}\BG(s)$. 
    Thus, both $\BGforEA(s)$ and $\BGforB(s)$ are equivalent to $\left[\left(\BW(-s)^{\trans}\right)^{-1}\BG_\infty(s)\right]_+$, as claimed in~\cref{eq:GforA_E_bst}.
\end{proof}

In conjunction, Theorem~\ref{thm:gen_quadbt} and Theorem~\ref{thm:bst_from_data} provide the theoretical foundation for a data-driven formulation of \BST.
We call this approach \emph{quadrature-based \BST} (\QBST).
Algorithm~\ref{alg:gen_quadbt} yields \QBST when the transfer functions $\BGforC(s)$ and $\BGforEA(s)$, $\BGforB(s)$ to be sampled are chosen as in~\cref{eq:GforB_C_bst} and \cref{eq:GforA_E_bst}, respectively.
We emphasize that these results form the theoretical formulation for \QBST{}.
Computationally, it is not clear how to measure the spectral factor $\BW(s)$ outside of intrusive numerical evaluation using the state-space matrices $\BA, \BB,\BC,$ and $\BD$, and a computed minimal solution $\BQbst^{-}$ to~\cref{eq:bst_riccati}.
Still, the transfer functions in Theorem~\ref{thm:bst_from_data} are input-output invariants; once these data are computed, the resulting data-based reduced model computed by Algorithm~\ref{alg:gen_quadbt} does not refer to any internal quantities.
This will be the case for the data-driven formulations of \PRBT and \BRBT that we derive next, as well.
\subsection{\PRBT from data: \QPRBT}
\label{ss:qprbt}

Again, suppose that $\Sys$ is square. A closely related method to that of \BST is \emph{positive-real balanced truncation} (\PRBT), also introduced in~\cite{green1988balanced}.
Passive systems are ubiquitous in applications of physics and engineering; electrical circuits are one such example.
Passive systems also have \emph{port-Hamiltonian} structure, an important focus of recent work in the modeling community~\cite{mehrmann2023control}. 
For $s\in\C$ that is not an eigenvalue of either $\BA$ or $-\BA$, the \emph{Popov function} of $\Sys$ is defined as
\begin{equation}
\label{eq:popov}
    \Phi(s) \coloneqq \BG(s) + \BG(-s)^{\trans},~~s\in\C.
\end{equation}
A system being passive is equivalent to its transfer function being \emph{positive real}~\cite[Theorem~5.30]{antoulas2005approximation}, i.e.,
\begin{align}
    \label{eq:pr_con}
    \Phi(i\omega) = \BG(\imunit \omega) + \BG(-\imunit\omega)^{\trans} \succeq0, \quad \omega\in\R.
\end{align}
$\Sys$ is \emph{strictly} positive-real if the inequality~\cref{eq:pr_con} is strict. 
For simplicity, we consider only strictly positive-real systems in this work, and take positive-realness to mean \emph{strict} positive-realness moving forward.
\PRBT is primarily utilized for the model reduction of \emph{passive} linear systems. 
\PRBT-\ROM{}s are guaranteed to maintain the asymptotic stability \emph{and} passivity of the \FOM.
Additionally, there exists a relative type of bound on the $\CH_\infty$ reduction error; see~\cite[Lemma~3]{gugercin2004survey}.

Suppose that $\BR\coloneqq\BD+\BD^{\trans}\succ 0$. By~\cite[Corollary~13.27]{zhou1996robust},
an asymptotically stable linear system $\Sys$ is strictly positive real if and only if there exists a (\SPD) stabilizing solution $\BQprbt\in\Rnn$ to the \ARE
\begin{align}
\begin{split}
\label{eq:prbt_obsv_riccati}
    \BA&^{\trans}\BQprbt + \BQprbt\BA + \\
    &(\BC-\BB^{\trans}\BQprbt)^{\trans}(\BD+\BD^{\trans})^{-1}(\BC-\BB^{\trans}\BQprbt)=\Bzero.
\end{split}
\end{align}
From the dual statement of~\cite[Corollary~13.27]{zhou1996robust}, there exists an \SPD stabilizing solution $\BPprbt\in\Rnn$ to the dual \ARE{}
\begin{align}
\begin{split}
    \label{eq:prbt_reach_riccati}
    \BA&\BPprbt + \BPprbt\BA^{\trans} + \\
    &(\BB-\BPprbt\BC^{\trans})(\BD+\BD^{\trans})^{-1}(\BB-\BPprbt\BC^{\trans})^{\trans} =\Bzero,
\end{split}
\end{align}
We refer to~\cref{eq:prbt_obsv_riccati} and~\cref{eq:prbt_reach_riccati} as the \emph{positive-real \ARE{}s} (\PRARE{}s).
The relevant Gramians in \PRBT{} are given by the \emph{minimal} solutions $\BQprbt^{-}\in\Rnn$ and $\BPprbt^{-}\in\Rnn$ to the \PRARE{}s~\cref{eq:prbt_obsv_riccati} and~\cref{eq:prbt_reach_riccati}.

For the data-based formulation of \PRBT{} that we derive, we require spectral factorizations of the Popov function~\cref{eq:popov}.
By~\cite[Corollary~13.27]{zhou1996robust}, the Popov function $\Phi(s)$ in~\cref{eq:popov} has a minimum phase left spectral factor $\BM(s)\in\Cmm$, i.e.
\begin{align*}
\Phi(s) = \BG(s) + \BG(-s)^{\trans} = \BM(-s)^{\trans}\BM(s),
\end{align*}
for all $s\in\C$.
The spectral factor $\BM(s)$ is the rational transfer function of the asymptotically stable and minimal linear system 
\begin{align}
\label{eq:prbt_M_spectral_fact}
    \CM&=\left(\BA, \BB, \BC_\CM, \BR^{1/2}\right)\\
\label{eq:prbt_M_spectral_fact_tf}
    \mbox{with}~~\BM(s)&=\BC_{\CM}(s\BI_n-\BA)^{-1}\BB + \BR^{1/2},
\end{align}
where $\BR=\BD+\BD^{\trans}$. The output matrix $\BC_\CM\in\Rmm$ of $\CM$ is defined as
\begin{align}
    \label{eq:Cm}
    \BC_{\CM}\coloneqq&~\BR^{-1/2}\left(\BC-\BB^{\trans}\BQprbt^{-}\right), 
\end{align}
where $\BQprbt^{-}$ is the minimal solution of~\cref{eq:prbt_obsv_riccati}.
Applying~\cite[Corollary~13.27]{zhou1996robust} to the dual of $\Sys$, one obtains a minimal phase right spectral factor $\BN(s)$ of $\Phi(s)$, i.e.,
\begin{align*}
\Phi(s) = \BG(s) + \BG(-s)^{\trans} = \BN(s)\BN(-s)^{\trans},
\end{align*}
for all $s\in\C$.
The spectral factor $\BN(s)$ is the rational transfer function of the asymptotically stable and minimal linear system 
\begin{align}
    \label{eq:prbt_N_spectral_fact}
    \CN&=\left(\BA, \BB_\CN, \BC, \BR^{1/2}\right)\\
    \label{eq:prbt_N_spectral_fact_tf}
    \mbox{with}~~\BN(s)&=\BC(s\BI_n-\BA)^{-1}\BB_\CN + \BR^{1/2}.
\end{align}
The input matrix $\BB_\CN\in\Rmm$ of $\CN$ is defined as
\begin{align}
    \label{eq:Bn}
    \BB_{\CN}\coloneqq\left(\BB-\BPprbt^{-}\BC^{\trans}\right)\BR^{-1/2},
\end{align}
where $\BPprbt^{-}$ is the minimal solution of~\cref{eq:prbt_reach_riccati}.
With these spectral factors, we can now apply the \GQBT{} framework of Section~\ref{sec:gen_quadbt} to derive a data-driven formulation of \PRBT{}.
Fixing $\BQprbt^{-}$ and $\BPprbt^{-}$ in the right-hand side of~\cref{eq:prbt_obsv_riccati} and~\cref{eq:prbt_reach_riccati} yields
\begin{align}
    \label{eq:gen_obsv_lyap_prbt}
    \BA^{{\trans}}\BQprbt + \BQprbt\BA +
    \BC_{\CM}^{\trans}\BC_{\CM}&=\Bzero\\
    \label{eq:gen_reach_lyap_prbt}
    \mbox{and}\quad\BA\BPprbt + \BPprbt\BA^{{\trans}} +
    \BB_{\CN}\BB^{{\trans}}_{\CN}&=\Bzero
    ,
\end{align}
which are the \ALE{}s in~\cref{eq:gen_obsv_lyap} and~\cref{eq:gen_reach_lyap} for the choice of $\BCgenY=\BC_\CM$ in~\cref{eq:Cm} and $\BBgenX=\BB_\CN$ in~\cref{eq:Bn}.
The \ALE{}s~\cref{eq:gen_obsv_lyap_prbt} and~\cref{eq:gen_reach_lyap_prbt} are uniquely solved by $\BQprbt^{-}$ and $\BPprbt^{-}$.
Equation~\cref{eq:gen_obsv_lyap_prbt} is the \emph{observability Lyapunov equation} of the linear system $\CY=\CM$ in~\cref{eq:prbt_M_spectral_fact} associated with the spectral factor $\BM(s)$, and $\BQgen=\BQprbt^{-}$ is thus the \emph{observability Gramian} of $\CM$.
Similarly, 
equation~\cref{eq:gen_reach_lyap_prbt} is the \emph{reachability Lyapunov equation} of the linear system $\CX=\CN$ in~\cref{eq:prbt_N_spectral_fact} associated with the spectral factor $\BN(s)$, and $\BPgen=\BPprbt^{-}$ is thus the \emph{reachability Gramian} of $\CN$.
 
Because $\BPgenX=\BPprbt^{-}$ and $\BQgenY=\BQprbt^{-}$ uniquely satisfy the \ALE{}s~\cref{eq:gen_reach_lyap_prbt} and~\cref{eq:gen_obsv_lyap_prbt}, these Gramians have integral representations~\cref{eq:reach_gram} and~\cref{eq:obsvgram}, and can be decomposed into the quadrature-based square-root factors~\cref{eq:gen_cholU} and~\cref{eq:gen_cholL} for $\BBgenX=\BB_\CN$ and $\BCgenY=\BC_\CM$ in~\cref{eq:Bn} and~\cref{eq:Cm}.
Theorem~\ref{thm:gen_quadbt} and Algorithm~\ref{alg:gen_quadbt}
can thus be applied in this instance to yield a data-driven formulation of \PRBT.
Similar to \BST, the transfer functions $\BGforEA(s)$, $\BGforB(s)$, and $\BGforC(s)$ can be understood in terms of the spectral factors $\BM(s)$ and $\BN(s)$ of the Popov function associated with $\BG(s)$.

\begin{theorem}
    \label{thm:prbt_from_data}
    Let $\BQprbt^{-}\in\Rnn$ and $\BPprbt^{-}\in\Rnn$ be the minimal solutions to~\cref{eq:prbt_obsv_riccati} and~\cref{eq:prbt_reach_riccati}.
    Then, the transfer functions $\BGforEA(s),$ $\BGforB(s)$, and $\BGforC(s)$ defined in~\cref{eq:GforEA}---\cref{eq:GforC} of Theorem~\ref{thm:gen_quadbt} {(and Algorithm~\ref{alg:gen_quadbt})} are given by:
    \begin{align}
        \BGforEA(s) &= \left[\left(\BM(-s)^{\trans}\right)^{-1}\BN_\infty(s)\right]_+, \label{eq:GforA_E_prbt}\\ 
        \BGforB(s) &= \BM_\infty(s),\quad
        \BGforC(s) = \BN_\infty(s), \label{eq:GforB_C_prbt}
    \end{align}
    where $\BM(s)$ and $\BN(s)$ are defined as in~\cref{eq:prbt_M_spectral_fact_tf} and~\cref{eq:prbt_N_spectral_fact_tf}.
\end{theorem}

\begin{proof}{Proof of Theorem~\ref{thm:prbt_from_data}.}
    In this case, $\BB_\BX = \BB_\CN,$ and $\BCgenY=\BC_\CM$ as in~\cref{eq:Bn} and~\cref{eq:Cm}. 
    So, from~\cref{eq:GforB} and~\cref{eq:GforC}
    \begin{align*}
        \BGforB(s)&=\BC_\CM(s\BI_n-\BA)^{-1}\BB=\BM_\infty(s),\\
        \BGforC(s)&=\BC(s\BI_n-\BA)^{-1}\BB_\CN=\BN_\infty(s),
    \end{align*}
    thus proving~\cref{eq:GforB_C_prbt}.
    The claim in~\cref{eq:GforA_E_prbt} follows nearly identically from the argument of Theorem~\ref{thm:bst_from_data} by replacing $\BW(s)$ with $\BM(s)$ and $\BG_\infty(s)$ with $\BN_\infty(s)$.
\end{proof}

As was the case for \BST, Theorem~\ref{thm:gen_quadbt} and Theorem~\ref{thm:prbt_from_data} provide the theoretical foundation for a data-driven implementation of \PRBT, that we call \emph{quadrature-based \PRBT} (\QPRBT). Algorithm~\ref{alg:gen_quadbt} yields \QPRBT when 
the transfer functions $\BGforEA(s)$ and $\BGforB(s)$, $\BGforC(s)$ to be sampled are replaced with those in~\cref{eq:GforA_E_prbt} and \cref{eq:GforB_C_prbt}.

\subsection{\BRBT from data: \QBRBT}
\label{ss:qbrbt}
The last variant studied here is \emph{bounded-real balanced truncation} (\BRBT)~\cite{opdenacker1988contraction}.
An important class of systems is those having transfer functions that are bounded along the imaginary axis; such systems are used in parameterizing all stabilizing controllers of a system such that the closed-loop system satisfies a particular $\CH_\infty$ constraint~\cite{glover1988state}. These systems are called \emph{bounded-real}. Formally, an asymptotically system $\Sys$ is called bounded-real if its transfer function $\BG(s)$ satisfies
\begin{align}
    \label{eq:br_con}
    \gamma^2\BI_m - \BG(-\imunit\omega)^{\trans}\BG(\imunit\omega) \succeq 0,\quad \omega\in\R,
\end{align}
where $\gamma \coloneqq \|\Sys\|_{\CH_\infty}$. Condition~\cref{eq:br_con} can alternatively be posed as $\|\Sys\|_{\CH_\infty}\leq \gamma$. Since it is always possible to normalize $\BG(s)$ so that $\|\Sys\|_{\CH_\infty}\leq 1$, without loss of generality, we assume that $\gamma=1$ in this work.
A system is \emph{strictly} bounded-real if the inequality~\cref{eq:br_con} is strict; moving forward, we assume strict bounded-realness, and take bounded-realness to mean strict bounded-realness. \BRBT-\ROM{}s preserve the asymptotic stability and bounded-realness of the \FOM, as well as satisfy an \emph{a priori} error bound on the absolute $\CH_\infty$ error~\cite{opdenacker1988contraction}.

Suppose that $\BI_m-\BD^{\trans}\BD \succ 0$.
According to the Bounded-real Lemma~\cite[Corollary~13.24]{zhou1996robust}, $\Sys$ is strictly bounded-real if and only if there exists a \SPD stabilizing solution $\BQbrbt\in\Rnn$ to the \ARE
\begin{align}
\begin{split}
    \label{eq:brbt_obsv_riccati}
    \BA^{\trans}&\BQbrbt + \BQbrbt\BA + \BC^{\trans}\BC +
    \left(\BB^{\trans}\BQbrbt+\BD^{\trans}\BC\right)^{\trans}\times\\
    &\left(\BI_m-\BD^{\trans}\BD\right)^{-1}\left(\BB^{\trans}\BQbrbt+\BD^{\trans}\BC\right) =\Bzero.
\end{split}
\end{align}
Suppose now that $\BI_p-\BD\BD^{\trans}\succ 0$.
By the dual result of~\cite[Corollary~13.24]{zhou1996robust}, there exists a
\SPD stabilizing solution $\BPbrbt\in\Rnn$
 \begin{align}
 \begin{split}
    \label{eq:brbt_reach_riccati}
    \BA&\BPbrbt + \BPbrbt\BA^{\trans} + \BB\BB^{\trans} +\left(\BPbrbt\BC^{\trans} + \BB\BD^{\trans}\right)\times\\
    &\left(\BI_p-\BD\BD^{\trans}\right)^{-1}\left(\BPbrbt\BC^{\trans} + \BB\BD^{\trans}\right)^{\trans}=\Bzero.
\end{split}
\end{align}
We refer to~\cref{eq:brbt_obsv_riccati} and~\cref{eq:brbt_reach_riccati} as the \emph{bounded-real \ARE{}s} (\BRARE{}s).
Similar to \PRBT{}, the relevant Gramians in \BRBT{} are given by the \emph{minimal} solutions $\BQbrbt^{-}\in\Rnn$ and $\BPbrbt^{-}\in\Rnn$ to the \BRARE{}s~\cref{eq:brbt_obsv_riccati} and~\cref{eq:brbt_reach_riccati}.

Suppose $\Sys$ is strictly bounded-real. Define the matrices $\BRj\coloneqq\BI_m - \BD^{\trans}\BD\in\Rmm$ and $\BRk\coloneqq\BI_p - \BD\BD^{\trans}\in\Rpp$, and assume that $\BRj$ and $\BRk$ are \SPD{}.
Let $\BQbrbt^{-}$ and $\BPbrbt^{-}$ be the minimal solutions to the \BRARE{}s~\cref{eq:brbt_obsv_riccati} and~\cref{eq:brbt_reach_riccati}, respectively.
By~\cite[Corollary~13.21]{zhou1996robust}, there exists $\BJ(s)\in\Cmm$ such that $\BJ(s)$ is a minimal phase left spectral factor of $\BI_m-\BG(-s)^{\trans}\BG(s)$, i.e.
$$\BI_m-\BG(-s)^{\trans}\BG(s)=\BJ(-s)^{\trans}\BJ(s),$$
for all $s\in \C$
The spectral factor $\BJ(s)$ is the rational transfer function of the asymptotically stable and minimal system $\CJ$ determined by the quadruple $\big(\BA,\BB,\BCbrbt,\BRj^{1/2}\big)$, where
\begin{align}
    \label{eq:Cs}
   \BCbrbt\coloneqq&-\BRj^{-1/2}\left(\BB^{\trans}\BQbrbt^{-} + \BD^{\trans}\BC\right),\\
    \label{eq:S_spectral_fact}
    \BJ(s) =&~\BCbrbt(s\BI_n-\BA)^{-1}\BB+\BRj^{1/2},
\end{align}
and $\BQbrbt^{-}$ is the minimal solution of~\cref{eq:brbt_obsv_riccati}.
Similarly, by the dual result~\cite[Corollary~13.22]{zhou1996robust},
there exists a minimal phase right spectral factor $\BK(s)\in\Cpp$ of $\BI_p-\BG(s)\BG(-s)^{\trans}$, i.e.,
$$\BI_p-\BG(s)\BG(-s)^{\trans}=\BK(s)\BK(-s)^{\trans},$$
for $s\in\C$.
$\BK(s)$ is the transfer function of the asymptotically stable and minimal system $\CK$ determined by the quadruple $\big(\BA,\BBbrbt,\BC,\BRk^{1/2}\big)$, where
\begin{align}
    \label{eq:Bz}
    \BBbrbt\coloneqq&-\left(\BPbrbt^{-}\BC^{\trans} + \BB\BD^{\trans}\right)\BRk^{-1/2},\\
\label{eq:Z_spectral_fact}
\BK(s)=&~\BC(s\BI_n-\BA)^{-1}\BBbrbt+\BRk^{1/2},
\end{align}
and $\BPbrbt^{-}$ is the minimal solution of~\cref{eq:brbt_reach_riccati}.
The \BRARE{}s~\cref{eq:brbt_obsv_riccati} and~\cref{eq:brbt_reach_riccati} can now be interpreted as the observability and reachability Lyapunov equations of some linear systems $\CY$ and $\CX$ in~\cref{eq:gen_sys}. Although, unlike the previous cases of \BST and \PRBT, these systems are \emph{not} $\CJ$ and $\CK$ corresponding to the previously introduced factors.
Define the matrices:
\begin{align}
\begin{split}
\begin{alignedat}{2}
\label{eq:stacked_matrices}
    \hatBBbrbt &\coloneqq\begin{bmatrix}\BB & \BBbrbt\end{bmatrix}\in \R^{n\times 2m},~~~&&\hatBRk \coloneqq\begin{bmatrix}\BD & \BRk^{1/2}\end{bmatrix}\in \R^{p\times 2m},\\
    \hatBCbrbt &\coloneqq\begin{bmatrix}\BC \\ \BCbrbt\end{bmatrix}\in \R^{2p\times n},~~~&&\hatBRj \coloneqq\begin{bmatrix}\BD \\ \BRj^{1/2}\end{bmatrix}\in \R^{2p\times m}.\\
\end{alignedat}
\end{split}
\end{align}
Then, by fixing $\BQbrbt^{-}$ and $\BPbrbt^{-}$ in the right-hand side of the \BRARE{}s,~\cref{eq:brbt_obsv_riccati} and~\cref{eq:brbt_reach_riccati} become
\begin{align}
    \label{eq:refact_brbt_obsv_lyap}
    \BA^{\trans}\BQbrbt + \BA\BQbrbt + \hatBCbrbt^{\trans}\hatBCbrbt&=\Bzero\\
    \label{eq:refact_brbt_reach_lyap}
    \mbox{and}\quad\BA\BPbrbt + \BA^{\trans}\BPbrbt + \hatBBbrbt\hatBBbrbt^{\trans}&=\Bzero,
\end{align}
which agree with~\cref{eq:gen_obsv_lyap} and~\cref{eq:gen_reach_lyap} for the choice of $\BC_\CY=\hatBCbrbt$ and $\BB_\CX=\hatBBbrbt.$
Thus, we consider~\cref{eq:refact_brbt_obsv_lyap} to be the \emph{observability Lyapunov equation} of the asymptotically stable linear system $\CY=\wh\CJ$ defined by 
\begin{equation}
    \label{eq:brbt_Jhat_spectral_fact}
    \wh\CJ=\left(\BA,\BB, \hatBCbrbt, \hatBRj\right),
\end{equation}
and $\BQgen=\BQbrbt^{-}$ is the \emph{observability Gramian} of $\wh\CJ$. The transfer function $\wh\BJ(s)\in \C^{2p \times m}$ of $\wh\CJ$ is
\begin{align}
\label{eq:stacked_J}
    \wh\BJ(s)\coloneqq \hatBCbrbt(s\BI_n-\BA)^{-1}\BB + \hatBRj
        = \begin{bmatrix}
           \BG(s) \\ \BJ(s)
        \end{bmatrix}.
\end{align}    
Similarly,~\cref{eq:refact_brbt_reach_lyap} is the \emph{reachability Lyapunov equation} of the asymptotically stable linear system $\CX=\wh\CK$ defined by 
\begin{equation}
    \label{eq:brbt_Khat_spectral_fact}
    \wh\CK=\left(\BA,\hatBBbrbt, \BC, \hatBRk\right),
\end{equation}
and $\BPgen=\BPbrbt$ is the \emph{reachability Gramian} of $\wh\CK$.
The transfer function $\wh\BK(s)\in \C^{p \times 2m}$ of $\wh\CK$ is
\begin{align}
\label{eq:stacked_K}
   \hspace{-1ex} \wh\BK(s)\coloneqq \BC(s\BI_n-\BA)^{-1}\hatBBbrbt +\hatBRk
                = \begin{bmatrix}
                    \BG(s) & \BK(s)
                \end{bmatrix}.
\end{align}   
This fits the \GQBT framework of Section~\ref{sec:gen_quadbt}. 
By asymptotic stability, the relevant Gramians $\BPgenX=\BPbrbt^{-}$ and $\BQgenY=\BQbrbt^{-}$ admit contour integral representations~\cref{eq:reach_gram} and~\cref{eq:obsvgram}, and can be decomposed into the quadrature-based square-root factors~\cref{eq:gen_cholU} and~\cref{eq:gen_cholL} for $\BBgenX=\hatBBbrbt$ and $\BCgenY=\hatBCbrbt$.
Once more, this sets the stage for the application of Theorem~\ref{thm:gen_quadbt}; the transfer functions $\BGforEA(s)$, $\BGforB(s)$, and $\BGforC(s)$ can be understood in terms of $\BJ(s)$, $\wh\BJ(s)$, $\BK(s)$, and $\wh\BK(s)$.

\begin{theorem}
    \label{thm:brbt_from_data}
    Let $\BQbrbt^{-}\in\Rnn$ and $\BPbrbt^{-}\in\Rnn$ be the minimal solutions to~\cref{eq:brbt_obsv_riccati} and~\cref{eq:brbt_reach_riccati}.
    Then the transfer functions $\BGforEA(s)$,  $\BGforB(s)$, and $\BGforC(s)$ defined in~\cref{eq:GforEA}---\cref{eq:GforC} of Theorem~\ref{thm:gen_quadbt} {(and Algorithm~\ref{alg:gen_quadbt})} are given by:
    \begin{align}
        \begin{split}\BGforEA(s) &= \left[\begin{bmatrix}
            \BI_p & \Bzero_{p\times m}\\
            \BG_\infty(-s)^{\trans} & \BJ_\infty(-s)^{\trans}
        \end{bmatrix}^{-1} \times \right.\\
        &\quad\quad\quad
        \left.
        \begin{bmatrix}
            \BG_\infty(s) & \BK_\infty(s)\\
            -\BD^{\trans}\BG_\infty(s) & -\BD^{\trans}\BK_\infty(s)
        \end{bmatrix}\right]_+, 
        \end{split}\label{eq:GforA_E_brbt}\\
        \BGforB(s) &= \wh\BJ_\infty(s),~~~~
        \BGforC(s) = \wh\BK_\infty(s),\label{eq:GforB_C_brbt}
    \end{align}
    where $\BJ(s)$, $\wh{\BJ}(s)$, $\BK(s)$, and $\wh{\BK}(s)$ are defined as in~\cref{eq:S_spectral_fact},~\cref{eq:stacked_J},~\cref{eq:Z_spectral_fact}, and~\cref{eq:stacked_K}.
\end{theorem}

\begin{proof}{Proof of Theorem~\ref{thm:brbt_from_data}.}
    Here, $\BBgenX=\hatBBbrbt$ and $\BCgenY=\hatBCbrbt$ defined in~\cref{eq:stacked_matrices}. So, from~\cref{eq:GforB} and~\cref{eq:GforC} we have that
    \begin{align*}
        \BGforB(s)&= \hatBCbrbt(s\BI_n-\BA)^{-1}\BB= \wh\BJ_\infty(s)\\
        \mbox{and}~~\BGforC(s)&=\BC(s\BI_n-\BA)^{-1}\hatBBbrbt= \wh\BK_\infty(s),
    \end{align*}
    proving~\cref{eq:GforB_C_brbt}.
   Next, define the matrices
   $\wh\BB \in \R^{n\times (p + m)}$, $\wh\BC \in\R^{(p+m)\times n}$, and $\wh\BR \in\R^{(p+m)\times(p+m)}$ by
      \begin{align*}
       \wh\BB\coloneqq\begin{bmatrix}
           \Bzero_{n\times p} & \BB
       \end{bmatrix},~
       \wh\BC\coloneqq\begin{bmatrix}
           \BC\\
           -\BD^{\trans}\BC
       \end{bmatrix},~
       \wh\BR\coloneqq\begin{bmatrix}
           \BI_p & \\
           & \BRj^{1/2}
       \end{bmatrix}.
   \end{align*} 
   It is straightforward to verify that the linear systems corresponding to the transfer functions
    \[\begin{bmatrix}
            \BI_p & \Bzero_{p\times m}\\
            \BG_\infty(-s)^{\trans} & \BJ_\infty(-s)^{\trans}
        \end{bmatrix}^{-1},~\begin{bmatrix}
            \BG_\infty(s) & \BK_\infty(s)\\
            -\BD^{\trans}\BG_\infty(s) & -\BD^{\trans}\BK_\infty(s)
        \end{bmatrix},\]
        have realizations given by $(-\BA^{\trans}+\hatBCbrbt^{\trans}\wh\BR^{-1}\wh\BB^{\trans}$, $-\hatBCbrbt^{\trans}\BR^{-1}$, $\wh\BR^{-1}\wh\BB^{\trans}$, $\wh\BR^{-1})$ and $(\BA, \hatBBbrbt, \wh\BC, \Bzero_{p+m})$, where $\hatBCbrbt$ and $\hatBBbrbt$ are defined as in~\cref{eq:stacked_matrices}. 
        Using~\cite[Sec.~3.6]{zhou1996robust}, the cascaded system on the right-hand side of the claimed equality in~\cref{eq:GforA_E_brbt} has a realization
        \begin{align}
        \label{eq:cascade2}
        \left(\begin{array}{ccccc|ccc}
             & -\BA^{\trans} + \hatBCbrbt^{\trans}\wh\BR^{-1}\wh\BB^{\trans} &  & \hatBCbrbt^{\trans}\wh\BR^{-1}\wh\BC & & & \Bzero_{n\times(p+m)} &\\ 
            & \Bzero_n &  & \BA & & & \hatBBbrbt &\\ 
             \hline
            & \wh\BR^{-1}\BB^{\trans} &  &\wh\BR^{-1}\wh\BC & & & \Bzero_{p+m} &
        \end{array}\right).\end{align}
        After some manipulations, the \ARE~\cref{eq:brbt_obsv_riccati} can be rearranged to be written as
        \begin{align*}
            \left(-\BA^{\trans} + \hatBCbrbt^{\trans}\wh\BR^{-1}\wh\BB^{\trans}\right)\left(-\BQbrbt^{-}\right) + \BA\BQbrbt^{-}+ \hatBCbrbt^{\trans}\wh\BR^{-1}\wh\BC=\Bzero.
        \end{align*}
        Using the reformulation of ~\cref{eq:brbt_obsv_riccati} written above, the state-space transformation 
        $\BT = \begin{bmatrix}
            \hphantom{-}\BI_n\hphantom{-} & -\BQbrbt^{-}\hphantom{-}\\
            \Bzero_n & \BI_n
        \end{bmatrix}\in\R^{2n\times 2n}$
        decouples the cascaded system realization in~\cref{eq:cascade2}. In other words, the transformed state space is given by
        \begin{align} \label{eq:BRBT_cascade_diag}
        \left(\begin{array}{ccccc|ccc}
            & -\BA^{\trans} + \hatBCbrbt^{\trans}\wh\BR^{-1}\wh\BB^{\trans} &  & \Bzero_n & & & \BQbrbt^{-}\hatBBbrbt &\\ 
            & \Bzero_n &  & \BA & & & \hatBBbrbt &\\ 
             \hline
            & \wh\BR^{-1}\BB^{\trans} &  &\hatBCbrbt & & & \Bzero_{p+m} &
        \end{array}\right).\end{align}
        Evidently, the stable part of this system has the transfer function $\hatBCbrbt(s\BI_n-\BA)^{-1}\hatBBbrbt$. Recalling that $\BBgenX=\hatBBbrbt$, $\BCgenY=\hatBCbrbt$, this is precisely $\BGforEA(s)$ according to~\cref{eq:GforEA}.
\end{proof}

Theorem~\ref{thm:brbt_from_data} shows how to interpret the transfer functions in Theorem~\ref{thm:gen_quadbt} in the context of \BRBT.
We call the resulting data-driven formulation \emph{quadrature-based \BRBT} (\QBRBT).
Algorithm~\ref{alg:gen_quadbt} yields \QBRBT when the transfer functions $\BGforEA(s)$, and $\BGforB(s), \BGforC(s)$ to be sampled are chosen as in~\cref{eq:GforA_E_brbt} and \cref{eq:GforB_C_brbt}, respectively.

\section{Numerical Examples}
\label{sec:numerics}
In this section, we provide a numerical proof of concept for the data-driven \BTr-\ROM{}s derived in Sections~\ref{sec:gen_quadbt} and~\ref{sec:applied_genqbt}. For each of the variants \BST, \PRBT, and \BRBT, we compare the performance of the data-driven \ROM{}s computed via \QBST, \QPRBT, and \QBRBT following the layout of Algorithm~\ref{alg:gen_quadbt} against their respective intrusive counterparts via Algorithm~\ref{alg:sqrt_bt}. 
The transfer function data in Theorems~\ref{thm:bst_from_data}---\ref{thm:brbt_from_data} required to build the \GQBT-\ROM{}s are computed intrusively using the state-space matrices $\BA,\BB,\BC,$ and $\BD$ of the full-order model being approximated, and the minimal solutions to the relevant \ARE{}s in each case.
All experiments were performed on a MacBook Pro equipped with 8 gigabytes of RAM and a 2.3 GHz Dual-Core Intel Core i5 processor running macOS Ventura version 13.6.1.
The experiments were run using \MATLAB version 23.2.0.2428915 (R2023b) Update 4.
The \MATLAB code and data for reproducing all experiments are publicly available at~\cite{supRei23}.

\textbf{Experimental setup.}
The system under study is an order $n=400$ single-input single-output \textsf{RLC} circuit model $\Sys$ presented in \cite{gugercin2004survey}; the physical parameters are chosen as $R=C=L=0.1,$ and $\overline{R}=1$. The system is both passive and square by construction and contains a non-trivial and non-singular input feedthrough term $\BD$. 
For \QBRBT, we normalize the circuit model so that $\|\Sys\|_{\CH_\infty}=0.5$, and $\Sys$ satisfies the bounded-real assumption~\cref{eq:br_con}.
The \emph{intrusive} \BST, \PRBT, and \BRBT-\ROM{}s that we benchmark the \GQBT-\ROM{}s against are computed using the \MATLAB toolbox \MORLAB~\cite{benner2021morlab}.
In generating the {synthetic spectral factor} data {from Theorems~\ref{thm:bst_from_data}---\ref{thm:brbt_from_data}}, the built-in \MATLAB routine `\textsf{icare}' 
{was} used to compute the minimal solutions of the appropriate \ARE{}s.
To implicitly approximate the relevant Gramians $\BPgen$ and $\BQgen$ {in each instance}, we employ the Trapezoidal rule using $N=40, 80, 160$ quadrature nodes. 
These are chosen as logarithmically-spaced points in the interval $\imunit[10^{-1},10^{4}]\subset \imunit\R$, closed under complex conjugation.
Although any numerical quadrature that yields a sufficiently accurate approximation to the Gramians can be used; see~\cite[Prop.~3.2]{gosea2022data}.

 \begin{remark}
     In some practical scenarios, solely evaluations of $\BG(s)$ are available. A similar approach as in~\cite{BenGV20} can then be used: first, construct a reduced-order surrogate using these data via \QBT, and then use this surrogate to obtain numerical evaluations of the spectral factors required for \QBST, \QPRBT, and \QBRBT.
     We emphasize that this is just one possibility for non-intrusively computing the spectral factor data, and that we are not advocating for its use as a practical method.
     In a future work, we will consider computational strategies for computing the spectral factor data from samples of $\BG(s)$. 
 \end{remark}

Figures~\ref{fig:qbst_ex},~\ref{fig:qprbt_ex}, and~\ref{fig:qbrbt_ex} compare the performance of \QBST, \QPRBT, and \QBRBT, respectively, against that of their intrusive counterparts. In each figure, the top plot depicts the \emph{true} singular values of $\sigma(\BL_\CY^{\trans}\BU_\CX)$ against the \emph{data-driven} $\sigma(\wtL_\CY^*\wtU_\CX)$, and the bottom plot depicts the relative $\CH_\infty$ error $\|\Sys-\Sysred\|_{\CH_\infty}/\|\Sys\|_{\CH_\infty}$ induced by the intrusive and data-driven \ROM{}s.
As illustrated by the figures, for each \BTr-variant, the data-driven singular values capture the true dominant singular values accurately. Similarly, 
the data-driven \GQBT-\ROM{}s 
approach the approximation quality of their intrusive counterpart as the number of nodes $N$ increases.
Therefore, using only the relevant input-output data, we match the performance of the intrusive \BTr-\ROM{}s.

\begin{figure}[h!]
    \centering
    \includegraphics[width=\linewidth]{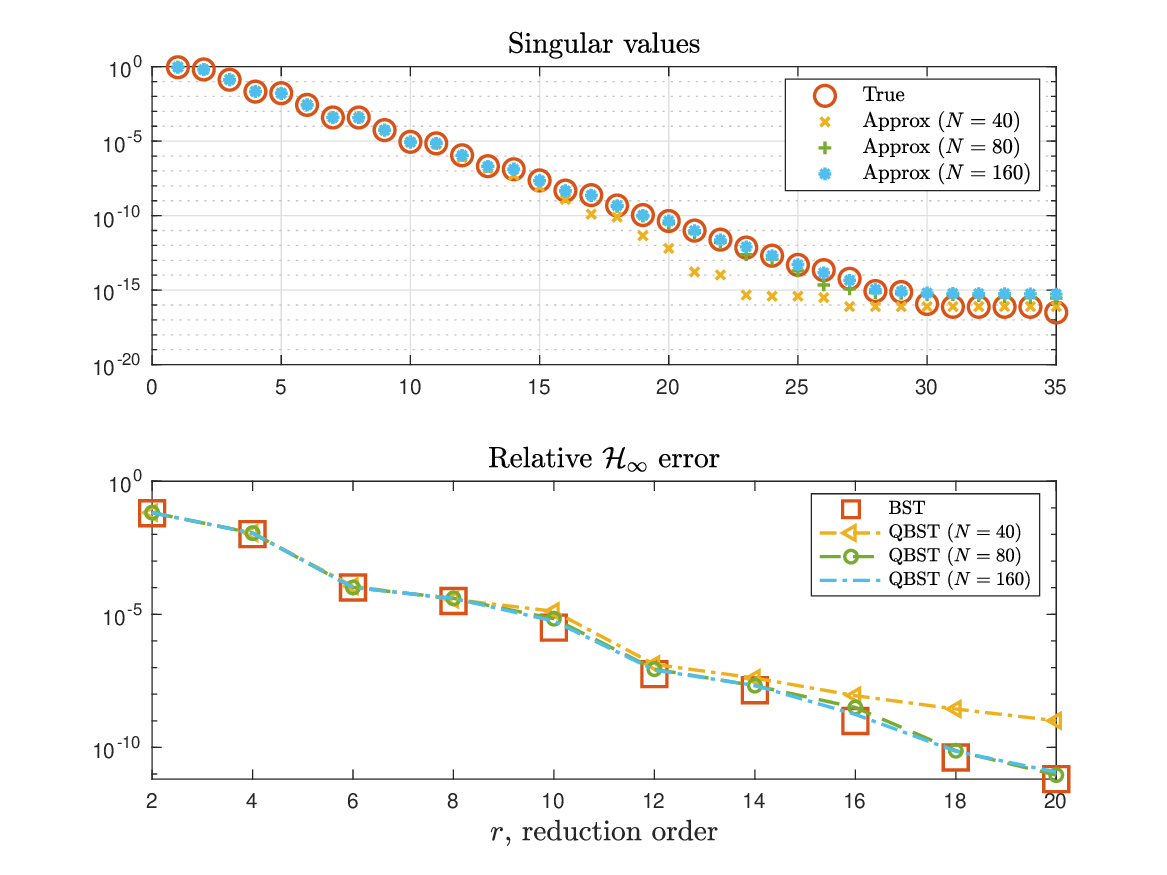}
    \caption{The true singular values compared against the approximate quadrature-based ones using $N$ quadrature nodes (top) and the relative $\CH_\infty$ approximation error for \BST and \QBST-\ROM{}s for orders $r=2,4,\ldots,20$ (bottom).}
    \label{fig:qbst_ex}
\end{figure}

\begin{figure}[h!]
    \centering
    \includegraphics[width=\linewidth]{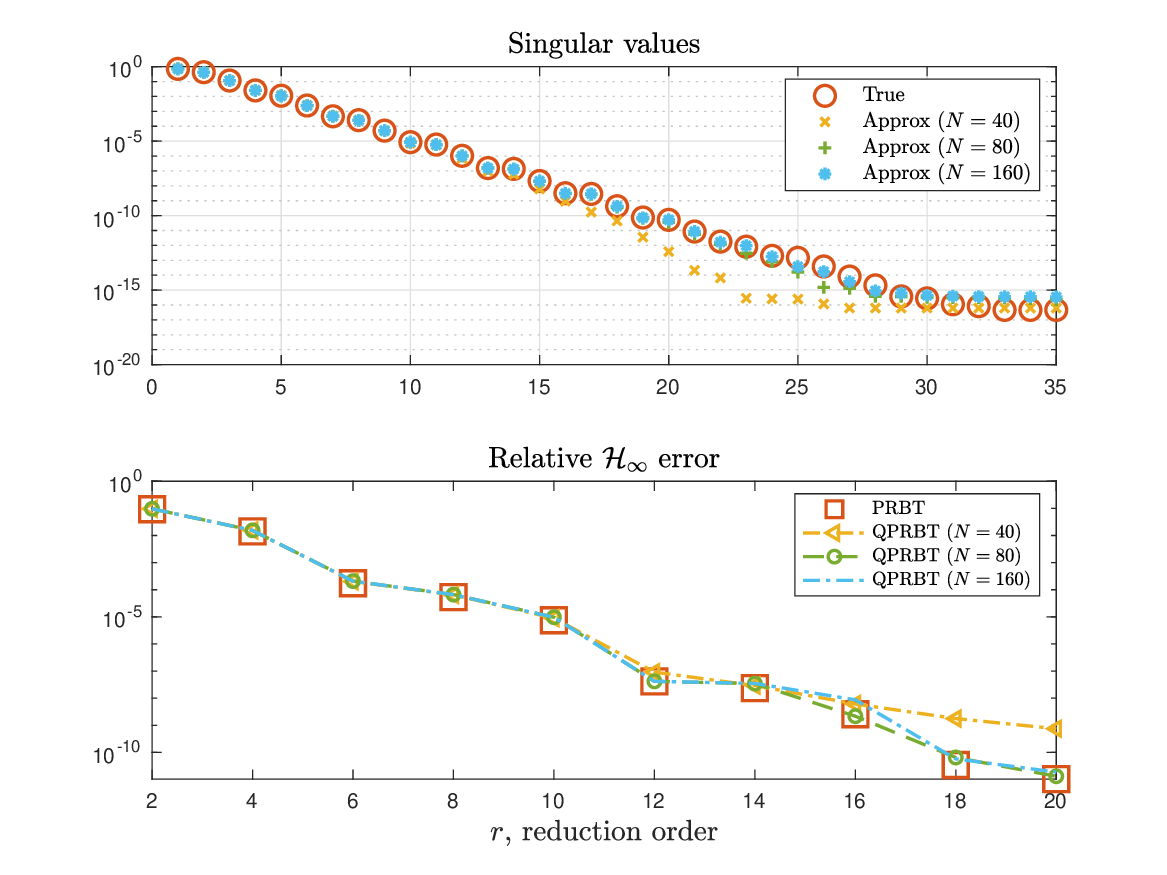}
    \caption{The true singular values compared against the approximate quadrature-based ones using $N$ quadrature nodes (top) and the relative $\CH_\infty$ approximation error for \PRBT and \QPRBT-\ROM{}s for orders $r=2,4,\ldots,20$ (bottom).}
    \label{fig:qprbt_ex}
\end{figure}

\begin{figure}[h!]
    \centering
    \includegraphics[width=\linewidth]{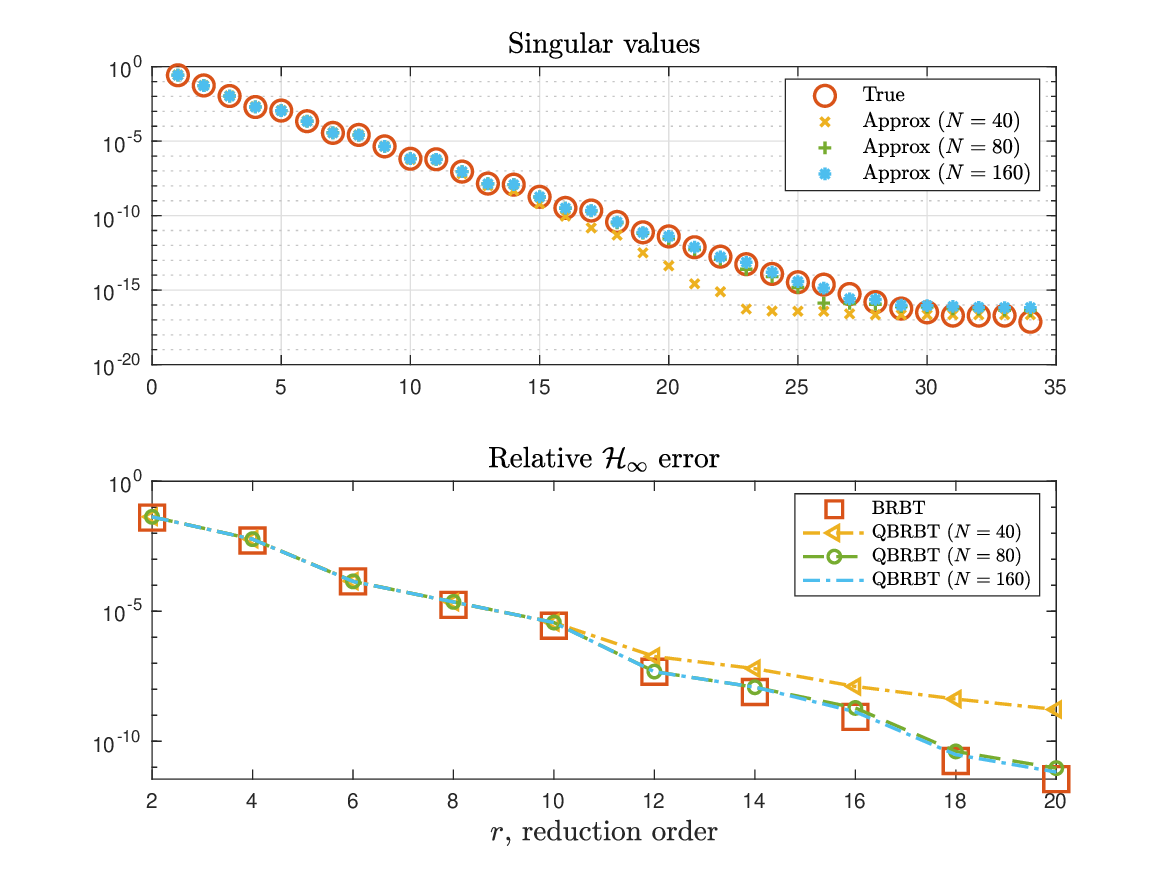}
    \caption{The true singular values compared against the approximate quadrature-based ones using $N$ quadrature nodes (top) and the relative $\CH_\infty$ approximation error for \BRBT and \QBRBT-\ROM{}s for orders $r=2,4,\ldots,{20}$ (bottom).}
    \label{fig:qbrbt_ex}
\end{figure}

\section{Conclusion}
\label{sec:conclusion}

We have developed the theoretical framework for data-driven formulations of \BST, \PRBT, and \BRBT by developing a generalized framework for quadrature-based balancing.
The resulting data-driven \BTr-\ROM{}s require sampling certain \emph{spectral factors} associated with the underlying model. 
The numerical examples serve as proof of concept for these data-driven \ROM{}s.
How to obtain the required samples of the relevant spectral factors in practice (experimentally) is an ongoing research question.

\section*{Acknowledgements}

This work was supported in part by the US NSF grants AMPS-
1923221 and DCSD-2130695.

\bibliographystyle{plainurl}
\bibliography{preprint}

\end{document}